\documentclass[11pt,reqno]{amsart}
\allowdisplaybreaks
\usepackage{hyperref}

\usepackage{amsmath}
\usepackage{amssymb}
\usepackage{mathrsfs}
\usepackage{amsthm}
\usepackage{bm}
\usepackage{booktabs}
\usepackage{float}
\usepackage{cases}
\usepackage{graphicx,epstopdf}
\usepackage[caption=false]{subfig}

\usepackage{geometry}
\usepackage{color}

\theoremstyle{plain}
\newtheorem{lemma}{Lemma}[section]
\newtheorem{theorem}[lemma]{Theorem}

\newtheorem{remark}[lemma]{Remark}
\newtheorem{proposition}[lemma]{Proposition}

\newcommand{\bbE}{\mathbb E}

\newcommand{\<}{\langle}
\renewcommand{\>}{\rangle}

\makeatletter

\newcommand{\Rmnum}[1]{\expandafter\@slowromancap\romannumeral #1@}
\makeatother

\begin{document}
\title[Novel splitting schemes for stochastic Langevin equation]{
A new class of splitting methods that preserve ergodicity and exponential integrability for stochastic Langevin equation}
\author{Chuchu Chen, Tonghe Dang, Jialin Hong, Fengshan Zhang*}
\address{LSEC, ICMSEC,  Academy of Mathematics and Systems Science, Chinese Academy of Sciences, Beijing 100190, China,
\and 
School of Mathematical Sciences, University of Chinese Academy of Sciences, Beijing 100049, China}
\email{chenchuchu@lsec.cc.ac.cn; dangth@lsec.cc.ac.cn; hjl@lsec.cc.ac.cn; zhangfs23@lsec.cc.ac.cn}
\thanks{This work is funded by the National key R\&D Program of China under Grant (No. 2020YFA0713701), National Natural Science Foundation of China (No. 12031020), and by Youth Innovation Promotion Association CAS, China.}
\thanks{*Corresponding author.}

\begin{abstract} 
In this paper, we propose a new class of  
splitting methods to solve the stochastic Langevin equation, which can simultaneously preserve the ergodicity  and  exponential integrability of the original equation. 
The central idea is to extract a stochastic subsystem that possesses the strict dissipation from the original equation,   
which is inspired by the inheritance of the Lyapunov structure for obtaining the ergodicity. 
We prove that the  exponential moment of the numerical solution is bounded, thus validating the exponential integrability of the proposed methods.  
Further, we show that under  moderate verifiable conditions, the  methods have the first-order convergence in both strong and weak senses,  
and  we present several concrete splitting schemes based on the  methods. 
The splitting strategy of methods can be readily extended to construct conformal symplectic methods and high-order  methods that preserve both the ergodicity and the exponential integrability,  
as demonstrated  in numerical experiments.  
Our numerical experiments also show that the proposed methods have good performance  in the long-time simulation.



 \end{abstract}
\keywords {Stochastic Langevin equation $\cdot$ Splitting  method, Ergodicity $\cdot$ Exponential integrability $\cdot$ Convergence}

\maketitle

\section{Introduction}\label{s0}
In this paper, we  study the construction and numerical analysis of splitting  methods that preserve intrinsic  properties of  the following stochastic Langevin equation, 
\begin{equation}\label{eq:0}
\left\{\begin{aligned}
    dP(t) &= -\upsilon P(t)dt-\nabla U(Q(t)) dt+\sigma dW_t,
    \\
    dQ(t) &= P(t)dt,
\end{aligned} \right.
\end{equation}
where the  initial
value $(P(0), Q(0))\in \mathbb{R}^{2}$ is deterministic, $\upsilon>0$ is the  friction coefficient,  $\sigma>0$ is the diffusion  coefficient, $U>0$ is the potential function which may exhibit  superquadratic growth,  and $W_t$ is a $1$-dimensional standard Wiener process on a filtered
complete probability space $(\Omega, \mathcal{F},  \{\mathcal{F}_t\}_{t\geq 0}, \mathbb{P})$.  For simplicity, we consider the $\mathbb R^2$-valued solution of \eqref{eq:0}, and in fact, our results hold 
for the general $\mathbb R^{2d}$ case.  
The dynamical variables $Q(t)$ and $P(t)$ denote the position and momentum  of
a Hamiltonian system with energy function
\begin{align}\label{energy1}H_0(p,q)=\frac{p^2}{2}+U(q).
\end{align}  
 Stochastic Langevin equation has wide application in  many
fields, such as chemical interactions, molecular simulations and quantum systems, see e.g. \cite{Coffey2012,Gillespie2000}. For instance, it is a fundamental   model for describing the behavior of microscopic particles in statistical physics systems, which captures the dynamic characteristics of
particles under the influence of random collisions from surrounding molecules. 

\subsection{Ergodicity and exponential integrability}
It  is known that the dynamical system generated by \eqref{eq:0} is   ergodic, i.e., for all smooth test functions $g,$
\begin{align*}
\lim_{T\to\infty}\frac1T\int_0^T\mathbb E[g(P(t),Q(t))]dt=\int_{\mathbb R^2}g(p,q)\pi(dp,dq)\text{ in }L^2(\mathbb R^2;\pi),
\end{align*} 
which is proved by the uniform moment boundedness  and the H\"ormander condition for the  solution. 
Here,  $\pi$ is the unique invariant measure, which is characterized by the  Gibbs density function, $
\pi(dp,dq)=\frac1Ze^{-\frac{2\upsilon}{\sigma^2} H_0(p,q)}dpdq,
$ 
where $Z$ is a normalization constant to ensure $\int_{\mathbb R^2}\pi(dp,dq)=1;$ see e.g. \cite{PavGri2014}.  Ergodicity describes the unity between state space and time, implying  that the long-time behavior of the stochastic process can be effectively captured by the invariant measure.  
The integral $\int_{\mathbb R^2}g(p,q)\pi(dp,dq)=:\pi(g)$, 
often referred to as the ergodic limit, is closely related to some important macroscopic physical quantities 
in practical applications. Thus,  approximating the ergodic limit becomes  crucial  for predicting these physical quantities.  
A fundamental approach to approximating the ergodic limit is by constructing numerical methods that preserve the ergodicity of \eqref{eq:0},  
which ensures that the true statistical properties of the solution process can be  accurately reflected  by numerical solutions. 

Exponential integrability is an important  property that helps in establishing theoretical analysis of systems across various fields, including stability analysis and large deviation theory; see e.g., \cite{Cox2013} and references therein. 
When  the potential $U$ is a polynomial  of   even order, 
\eqref{eq:0} is proved to admit the exponential integrability, i.e., 
 \begin{align}\label{exp_H}
\sup_{t\in[0,T]}\mathbb E\Big[\exp\Big\{\frac{C_{1}H_0(P(t),Q(t))}{e^{\sigma^{2}t}}\Big\}\Big]\leq e^{C_{2}(T+1)+H_0(P_0,Q_0)}, 
\end{align} 
where $C_1, C_2>0$ are some  generic constants.  We give the detailed proof in Appendix \ref{s2-1}. 
 For stochastic differential equation with the  non-monotone type condition, for example, \eqref{eq:0} with $U$ being of superquadratic growth, it is important to  ensure  the preservation of exponential integrability for numerical methods.  This enables  
establishing  a positive strong convergence order for   the numerical solution; see e.g.  \cite{dai2023orderone,Hutzen2018}. 

Therefore, for \eqref{eq:0} with the superquadratically growing $U$, the study on numerical methods that preserve both the ergodicity and exponential integrability  is of great importance, as it provides an effective way  to accurately  capture  the essential dynamics of the original  system in both theoretical and practical contexts.  

\subsection{Existing works, motivation, and plan}
For the stochastic Langevin equation with the quadratically growing $U$, there have been many works on the construction and convergence analysis for ergodicity-preserving numerical methods;   see  e.g. \cite{Longtime1,  Longtime2, Hongbook2022,Hongbook2019,IMA16, Matting2002, PhyD07} and references therein. 
While for the superquadratic case, the analysis is more technical and  there are only a few works. 
A  pioneering work is \cite{Talay2002}, where the ergodicity and first-order weak convergence of the implicit Euler scheme are   established. 
For the strong convergence analysis, as previously mentioned, 
 exponential integrability of the numerical solution  becomes crucial for deriving  the convergence order.  
For example, authors in \cite{CuiSheng2022} prove  the exponential integrability of the splitting averaged vector field scheme  for \eqref{eq:0} and establish its strong convergence order. See also \cite{dai2023orderone, Hutzen2018} for the study of exponential integrability and strong convergence order  of a class of stopped increment-tamed Euler approximations for stochastic differential equations.

We aim to construct new class of numerical methods of \eqref{eq:0} that preserve both the ergodicity and exponential integrability based on  the splitting technique. Splitting techniques are widely recognized as useful tools not only  for constructing numerical methods that preserve the properties of the original equation, but also  for obtaining high-order numerical schemes; see e.g. \cite{Longtime2,IMA16} and references therein. In order to describe splitting methods in a convenient way, it is useful to introduce the generators
$\mathcal L_1:=-\nabla U(q) \nabla_p,\,\mathcal L_2:=p\nabla _q,\,\mathcal L_3:=-\upsilon p\nabla_p,$ and $\mathcal L_4:=\frac12\sigma^2\Delta_p.$ Then the generator of \eqref{eq:0} is $\mathcal L=\mathcal L_1+\mathcal L_2+\mathcal L_3+\mathcal L_4.$ We use $e^{\tau \mathcal L_i}$ to denote the one-step evolution operator  of the dynamics generated by $\mathcal L_i$ with step-size $\tau.$

As for \eqref{eq:0}, one usually splits the original equation as an integrable    deterministic Hamiltonian subsystem and a solvable  stochastic subsystem.  
 For example,  a  commonly used   splitting of   \eqref{eq:0} 
has the evolution operator $\mathcal P_{\tau}^{\mathcal L_1+\mathcal L_2,\mathcal L_3+\mathcal L_4}:=e^{\tau (\mathcal L_1+\mathcal L_2)}e^{\tau (\mathcal L_3+\mathcal L_4)},$ which determines the splitting solution $\{(P_{\tau}(t_{n}),Q_{\tau}(t_{n}))^{\top}\}_{n\in\mathbb N}$ with $t_n:=n\tau.$ 
Here, $e^{\tau(\mathcal L_1+\mathcal L_2)}$ is the one-step evolution operator  of the deterministic Hamiltonian subsystem  
\begin{align*}
d\binom{\bar P(t)}{\bar Q(t)}=\binom{-\nabla U(\bar Q(t)) dt}{\bar P(t)dt}, \quad \binom{\bar  P(t_{n})}{\bar  Q(t_n)}= \binom{ P_{\tau}(t_{n})}{ Q_{\tau}(t_{n})}, 
\end{align*}
 whose solution satisfies $H_0(\bar P(t_{n+1}),\bar Q(t_{n+1}))=H_0(\bar P(t_n),\bar Q(t_n)),$ 
and $e^{\tau(\mathcal L_3+\mathcal L_4)}$ represents that of the stochastic subsystem 
\begin{align*}
d\binom{\tilde P(t)}{\tilde Q(t)}=\binom{-\upsilon \tilde P(t)dt+\sigma dW_t}{0}, \quad \binom{\tilde P(t_{n})}{\tilde Q(t_n)}= \binom{\bar P(t_{n+1})}{\bar Q(t_{n+1})},
\end{align*} 
whose solution $(\tilde P(t_{n+1}),\tilde Q(t_{n+1}))^{\top}=(P_{\tau}(t_{n+1}),Q_{\tau}(t_{n+1}))^{\top}$ 
satisfies 
 that $\bbE[\tilde P(t_{n+1})^2]=e^{-\upsilon \tau}\bbE[\tilde P(t_n)^2]+\frac{\sigma^2}{2\upsilon}(1-e^{-2\upsilon \tau})$ and $\bbE[U(\tilde Q(t_{n+1}))]=\bbE[U(\tilde Q(t_n))].$ It is known that this splitting admits the exponential integrability (see e.g. \cite{CuiSheng2022}). 
The well-known Lyapunov structure  for obtaining  the ergodicity has the form of   \begin{align}\label{diss}
\bbE[H_0(P_{\tau}(t_{n+1}),Q_{\tau}(t_{n+1}))|\mathcal F_{t_n}]\leq \alpha H_0(P_{\tau}(t_n),Q_{\tau}(t_n))+\beta 
\end{align}
with  $\alpha\in(0,1)$ and $\beta>0$.  Since the deterministic Hamiltonian subsystem is energy-preserving, it is critical to possess the strict dissipation for the stochastic subsystem to obtain the ergodicity, which fails for this splitting as illustrated by red line in  Figure \ref{mo1}.

\begin{figure}[H]
\centering
\subfloat{
\begin{minipage}[t]{0.4\linewidth}
\centering
\includegraphics[height=4cm,width=6.0cm]{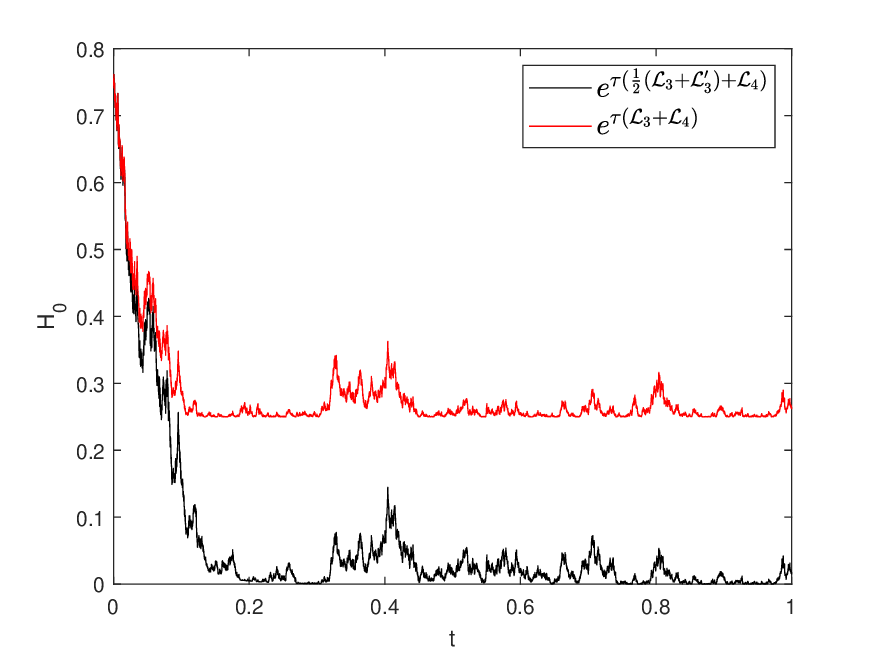}
\end{minipage}
}    
\centering
\caption{Evolution of  $H_0(\tilde P(t),\tilde Q(t))$.}\label{mo1}
\end{figure}

In order to make the stochastic subsystem be strictly dissipative such that  the resulted splitting system can inherit the Lyapunov structure, 
we introduce  a new operator $\mathcal L'_3:=-\upsilon q\nabla _q.$ Then we propose a novel splitting strategy as \begin{align}\label{newstra}\mathcal P_{\tau}^{\mathcal L_1+\mathcal L_2+\frac12(\mathcal L_3-\mathcal L'_3),\frac12(\mathcal L_3+\mathcal L'_3)+\mathcal L_4}:=e^{\tau(\mathcal L_1+\mathcal L_2+\frac12(\mathcal L_3-\mathcal L'_3))}e^{\tau(\frac12(\mathcal L_3+\mathcal L'_3)+\mathcal L_4)}.\end{align}  Precisely, we have 
\begin{align}\label{stra2}
d\binom{P(t)}{Q(t)} 
=\underset{\Psi^D}{\underbrace{\binom{-\frac{\upsilon}{2}P(t)-\nabla U(Q(t))}{P(t)+\frac{\upsilon}{2}Q(t)}dt}}+\underset{\Psi^S}{\underbrace{\binom{-\frac{\upsilon}{2}P(t)dt+\sigma dW_t}{-\frac{\upsilon }{2}Q(t)dt}}}, 
\end{align} 
 where $\Psi^D$ and $\Psi^S$ are flows for the deterministic subsystem and stochastic subsystem, respectively. This new splitting strategy has several advantages: \begin{itemize} \item[(\romannumeral1)] The deterministic subsystem $\Psi^D$ is still a Hamiltonian system  with the new Hamiltonian  
 \begin{align}\label{energy2}H(p,q):=H_0(p,q)+\frac{\upsilon }{2}pq. \end{align}
Moreover, the new Hamiltonian $H$ is equivalent to the original  energy function $H_0$ in a certain sense, i.e., there exists $C_0>0$ such that $C_1H_0\leq H+C_0\leq C_2(H_0+1)$ with some positive constants $C_1\leq C_2;$
\item[(\romannumeral2)] The stochastic subsystem $\Psi^S$ is strictly dissipative (see also black line in Figure \ref{mo1}) such that  the splitting system admits  the Lyapunov structure \eqref{diss} with the Hamiltonian $H$, $\alpha=e^{-\upsilon \tau},$ and some $\beta>0.$  Therefore, the splitting system is uniquely ergodic; see Proposition \ref{lem:2-1}.
\end{itemize} 

Based on the proposed splitting strategy \eqref{stra2}, we  construct a new class of splitting  methods that can simultaneously preserve the ergodicity and the exponential integrability of \eqref{eq:0}. 
To be specific,  
these splitting methods are obtained by applying the general conservative  methods to the deterministic part  and solving  the stochastic part  exactly.   
A key ingredient for the ergodicity of the  methods is the uniform moment boundedness of  numerical solutions over long  time.  
This is achieved  by the well balance between the preservation of the Hamiltonian for the deterministic
    subsystem and the dissipative property for the stochastic subsystem.   
In addition, this balance brings the exponential integrability of the  methods as well, controlling the exponential moment of numerical solution from explosion over finite time. 
Combining the stochastic Gr\"onwall inequality, 
we show that   the proposed methods have the first-order strong convergence under some moderate verifiable conditions on the one-step approximation.  Moreover, leveraging the Kolmogorov equation and the  uniform moment boundedness, we prove  the first-order convergence for the ergodic limit of the proposed methods. 


Several concrete splitting schemes based on our methods are presented  by incorporating the developed conservative  methods for  deterministic Hamiltonian systems. Then we perform  numerical experiments to verify our theoretical  result on strong and weak convergences in Section \ref{sec4.1}. In addition,   
as shown  by numerical experiments, the proposed splitting methods have the good ability in simulating the ergodic limit of the exact solution. 
As part of our investigation, based on our splitting strategy, we construct the second-order  methods that preserve both the ergodicity and the exponential integrability,  by employing the Strang splitting technique, and also present the conformal symplectic methods by applying  the symplectic methods for the deterministic Hamiltonian system, which    
 are demonstrated by  numerical experiments in Sections \ref{sec4.2} and \ref{sec4.3}.    
   

The rest of the paper is organized as follows. In Section \ref{scheme}, we  introduce the new class of  splitting  methods and show the inheritance of the ergodicity and the exponential integrability. 
In Section
\ref{order}, by using the moment properties of numerical solution, we prove that the strong and weak convergence orders are both $1$.  In section
\ref{s4}, we present       numerical experiments to verify our theoretical results.  
Appendix \ref{s2-1} is devoted to  some auxiliary proofs. 

Throughout this article, we use $C$ to denote a positive constant which
may not be the same in each occurrence. More specific constants which depend on
certain parameters $a,b,\ldots$ are numbered as $C(a,b,\ldots).$

\section{A new class of  splitting  methods}\label{scheme}
In this section, we aim to introduce   a new class of splitting methods that preserve both the ergodicity and the exponential integrability of \eqref{eq:0}. We take the potential $U(q):=\frac14q^4$ in this paper as an example to illustrate the main idea.   Some concrete numerical schemes based on our methods are also given.  In addition,  we present  the uniform moment boundedness, the ergodicity, and the exponential integrability of numerical solutions for proposed methods.  
\subsection{Construction} 

Let $\tau\in(0,1)$ be the uniform  time-step size and   $t_n:=n\tau,\,n\in \mathbb N$.
Recalling  \eqref{newstra}, \begin{align*}\mathcal P_{\tau}^{\mathcal L_1+\mathcal L_2+\frac12(\mathcal L_3-\mathcal L'_3),\frac12(\mathcal L_3+\mathcal L'_3)+\mathcal L_4}=e^{\tau(\mathcal L_1+\mathcal L_2+\frac12(\mathcal L_3-\mathcal L'_3))}e^{\tau(\frac12(\mathcal L_3+\mathcal L'_3)+\mathcal L_4)},\end{align*}   we split \eqref{eq:0} in the time interval $T_n:=[t_n,t_{n+1})$ into a deterministic Hamiltonian system with one-step evolution operator $e^{\tau(\mathcal L_1+\mathcal L_2+\frac12(\mathcal L_3-\mathcal L'_3))}$ and a linear stochastic differential equation with one-step  evolution operator $e^{\tau(\frac12(\mathcal L_3+\mathcal L'_3)+\mathcal L_4)}$. 
The solution of the splitting system is written as 
 \begin{align}\label{split}
 X_{\tau}(t_{n+1})=\Phi_{T_n}^S\Phi_{T_n}^D(X_{\tau}(t_n)),\quad n\in\mathbb N,
 \end{align}
 where $\{\Phi^D_{T_n}(t)\}_{t\in T_n}$ is the flow of the nonlinear Hamiltonian  system with random initial datum
 \begin{equation}\label{sub_ex1}
 \begin{cases}
 d\bar{P}_{T_n}(t) = -\frac{\upsilon}{2} \bar{P}_{T_n}(t)dt-{\bar{Q}_{T_n}}^3(t)dt,\quad \bar{P}_{T_n}(t_n)=P_{\tau}(t_n),
  \\
  d\bar{Q}_{T_n}(t) = \bar{P}_{T_n}(t)dt +\frac{\upsilon}{2}\bar{Q}_{T_n}(t)dt,\quad \bar{Q}_{T_n}(t_n)=Q_{\tau}(t_n), \end{cases}
 \end{equation}
 and $\{\Phi^S_{T_n}(t)\}_{t\in T_n}$ is the flow  of the linear stochastic differential equation
 \begin{equation}\label{sub_ex3}
 \begin{cases}
 d \tilde{P}_{T_n}(t) = -\frac{\upsilon}{2} \tilde{P}_{T_n}(t)dt +\sigma dW_t,\quad \tilde{P}_{T_n}(t_n)=\bar{P}_{T_n}(t_{n+1}),
   \\
    d \tilde{Q}_{T_n}(t) =-\frac{\upsilon}{2} \tilde{Q}_{T_n}(t) dt,\quad \tilde{Q}_{T_n}(t_n)=\bar{Q}_{T_n}(t_{n+1}), \end{cases}
 \end{equation}
 and $X_{\tau}(t_{n+1})=(P_{\tau}(t_{n+1}),Q_{\tau}(t_{n+1}))^{\top}=(\tilde P_{T_n}(t_{n+1}),\tilde Q_{T_n}(t_{n+1}))^{\top}.$  
 Especially, for $t\in T_0,$ the initial datum for \eqref{sub_ex1} is $(\bar P_{T_0}(0),\bar Q_{T_0}(0))=(P(0),Q(0)).$ It is observed that for any constant $C$,  $H+C$ is the Hamiltonian of \eqref{sub_ex1}.    

We apply the one-step  approximation  to \eqref{sub_ex1} which preserves the Hamiltonian and solve the linear subsystem  \eqref{sub_ex3} exactly to obtain a new class of splitting methods, the numerical solution $X_{n}:=(P_{n},Q_{n})^{\top}$ is defined recurrently by $(P_0, Q_0) := (P(0),Q(0)),$ and 
\begin{align}\label{SAVF}
X_{n+1}=\Phi^S_{T_n,t_{n+1}}\Upsilon^D_{\tau}(X_n),\quad n\in\mathbb N,
\end{align}
where $\Upsilon^D_{\tau}$ represents the one-step mapping defined by   
\begin{equation}\label{eq:2-1}
\begin{cases}
\bar{P}_{T_n,t_{n+1}} = P_n +\mathcal A_{\tau}^n,\\
\bar{Q}_{T_n,t_{n+1}} = Q_n+\mathcal B^n_{\tau},
\end{cases}
\end{equation}
and $\Phi^S_{T_n,t}:=(\tilde P_{T_n,t},\tilde Q_{T_n,t})^{\top}$  denotes the solution of \eqref{sub_ex3} at time $t\in T_n$  with initial datum $(\tilde P_{T_n,t_n},\tilde Q_{T_n,t_n})=(\bar P_{T_n,t_{n+1}},\bar Q_{T_n,t_{n+1}})$. 
Here,  $\mathcal A^n_{\tau}$ and $\mathcal B^n_{\tau}$ are measurable maps defined, respectively, as $\mathcal A^n_{\tau}:=\mathcal A_{\tau}(P_n,Q_n,\bar P_{T_n,t_{n+1}},\bar Q_{T_n,t_{n+1}})$ and $\mathcal B^n_{\tau}:=\mathcal B_{\tau}(P_n,Q_n,\bar P_{T_n,t_{n+1}},\bar Q_{T_n,t_{n+1}})$. In addition, without  loss of generality, the conservative method  \eqref{eq:2-1} is required to be solvable, namely, there exist measurable maps $\psi_{\tau},\phi_{\tau}:\mathbb R^2\to\mathbb R$ such that $\bar P_{T_n,t_{n+1}}=\psi_{\tau}(P_n,Q_n)$ and $\bar Q_{T_n,t_{n+1}}=\phi_{\tau}(P_n,Q_n).$

For the  Hamiltonian  system, there have been some conservative  methods 
 in the literature,  such as the
average vector field (AVF) method, the discrete-gradient (DG) method, 
and partitioned AVF (PAVF) method; see e.g. \cite{PAVF, CCD20,jinBIT, AVF,Tapley,zhouzhang} and references therein.   
  Hence, one can obtain a class of splitting  schemes  via  \eqref{SAVF}. Such  approach to  constructing numerical schemes has distinct advantages, especially in intrinsic property-preserving aspect. 
We will show this in the next subsection.

At the end of this subsection, we give some  concrete  examples of the splitting  method  \eqref{SAVF} based on the conservative  methods for the Hamiltonian system.  
\begin{itemize}
\item[(\romannumeral1)] 
When the one-step mapping $\Upsilon^D_{\tau}$ in   \eqref{eq:2-1} is defined by  the AVF scheme,  which reads as  
\begin{align*}
&\bar{P}_{T_n,t_{n+1}} = P_n -\frac{\tau\upsilon}{4}(\bar{P}_{T_n,t_{n+1}}+P_n)-\tau\int_0^1\left(Q_n+\lambda(\bar{Q}_{T_n,t_{n+1}}-Q_n)\right)^3d\lambda,
    \\
  &  \bar{Q}_{T_n,t_{n+1}} = Q_n+\frac{\tau}{2}(\bar{P}_{T_n, t_{n+1}}+P_n)+\frac{\tau\upsilon}{4}(\bar{Q}_{T_n,t_{n+1}}+Q_n),
\end{align*}
the corresponding splitting scheme \eqref{SAVF} is  called the splitting AVF (SAVF) scheme below.   \item[(\romannumeral2)] When the one-step mapping $\Upsilon^D_{\tau}$ in   \eqref{eq:2-1} is defined by the  DG scheme, 
\begin{align*}
 &   \bar{P}_{T_n,t_{n+1}}=P_n-\tau \bar{\nabla}_qH(\bar{P}_{T_n,t_{n+1}},\bar{Q}_{T_n,t_{n+1}};P_n,Q_n),
\\
 &   \bar{Q}_{T_n,t_{n+1}}=Q_n+\tau \bar{\nabla}_pH(\bar{P}_{T_n,t_{n+1}},\bar{Q}_{T_n,t_{n+1}};P_n,Q_n), 
\end{align*}
the corresponding splitting scheme \eqref{SAVF} is called the  splitting DG (SDG)  scheme below. Here, 
 $\bar{\nabla} H(\hat{p},\hat{q};p,q):=(\bar{\nabla}_pH(\hat{p},\hat{q};p,q); \bar{\nabla}_qH(\hat{p},\hat{q};p,q))$ is defined as 
$
    \bar{\nabla}H(\hat{p},\hat{q};p,q)=\nabla H(\bar{p},\bar{q})+\frac{H(\hat{p},\hat{q})-H(p,q)-\nabla H(\bar{p},\bar{q})^T\delta}{\|\delta\|^2}\delta,
$
    where $
    \delta = (\hat{p}-p;\hat{q}-q),\quad (\bar{p}; \bar{q})= \left(\frac{\hat{p}+p}{2};\frac{\hat{q}+q}{2}\right).
$
\item[(\romannumeral3)] When the one-step mapping $\Upsilon^D_{\tau}$ in   \eqref{eq:2-1} is defined by  the PAVF  scheme 
\begin{align*}
  \bar{P}_{T_n,t_{n+1}}&=P_n-\frac{\tau\upsilon}{2}\bar{P}_{T_n,t_{n+1}} -\tau\int_0^1(Q_n+\lambda(\bar{Q}_{T_n,t_{n+1}}-Q_n))^3d\lambda,
  \\
  \bar{Q}_{T_n,t_{n+1}}&=Q_n+\frac{\tau}{2}(\bar{P}_{T_n,t_{n+1}}+P_n)+\frac{\tau\upsilon}{2}Q_n,
\end{align*}
the corresponding splitting scheme \eqref{SAVF} is called  the splitting PAVF (SPAVF) scheme below. 
\end{itemize}
\subsection{Ergodicity and exponential integrability}
In this subsection, we   first give  the uniform moment boundedness  of the numerical solution for the proposed splitting method \eqref{SAVF}. Then we prove that the numerical solution is uniquely ergodic,  and admits the exponential integrability. These properties play an important role in obtaining the strong and weak  convergence orders  of the proposed methods.  Denote $n_s:=\max\{n\in\mathbb N:t_n\leq s\}$ for $s\ge 0.$ 
\begin{lemma}
\label{lem:2}
For any $p\geq 1$ 
and $T>0,$ there exists a constant $C:=C(p,T)$ such that 
\begin{align}\label{bound2}
\mathbb E\big[\sup_{t\in[0,T]}(|\tilde{P}_{T_{n_t},t}|^{2p}+|\tilde{Q}_{T_{n_t},t}|^{4p})\big]\leq C(1+H^p(P_0,Q_0)).
\end{align}
In addition, 
there exists a constant $C:=C(p)>0$ such that 
\begin{align}\label{bound1}
\sup_{n\in\mathbb N}\mathbb{E}[|P_{n}|^{2p}+|Q_n|^{4p}]\leq C(1+H^p(P_0,Q_0)).
\end{align}
\end{lemma}

The proof of Lemma \ref{lem:2} is postponed to Appendix \ref{s2-1}. The uniform moment boundedness of the numerical solution implies the existence of the numerical invariant measure. 
Furthermore, we can show that the numerical invariant measure is uniquely ergodic. 
Before that, it is observed that by the Young inequality, there exists a constant  $C_H>0$ such that $H+C_H\ge 1$. For instant, one can take $C_H=\frac{\upsilon^4}{64}+1,$ then using the Young inequality twice gives \begin{align}\label{HCH}
H(p,q)+C_H&\ge \frac{p^2}{2}+\frac{q^4}{4}-\frac{\upsilon}{4}(\frac{2p^2}{\upsilon}+\frac{\upsilon q^2}{2})+C_H\notag\\
&\ge \frac{q^4}{4}-\frac{\upsilon^2}{16}(\frac{4}{\upsilon^2}q^4+\frac{\upsilon^2}{4})+C_H\ge 1.
\end{align}
In addition, there exist constants $C,C_e>0$ such that \begin{align}\label{equ:1}
C_e(p^2+q^4)\leq H(p,q)+C\leq   C(p^2+q^4+1).  
\end{align}
In fact, this can be obtained by applying the  Young inequality as follows 
\begin{align*}
&H(p,q)\ge \frac{p^2}{2}+\frac{q^4}{4}-\frac{\upsilon}{4}(\frac{p^2}{\upsilon}+\upsilon q^2)\ge \frac{p^2}{4}+\frac14 q^4-\frac{\upsilon^2}{4}(\frac{1}{2\upsilon^2}q^4+\frac{\upsilon^2}{2})\ge \frac18(p^2+q^4)-\frac{\upsilon^4}{8},\\
&H(p,q)\leq \frac{p^2}{2}+\frac{q^4}{4}+\frac{\upsilon}{4}(p^2+\frac{q^4}{2}+\frac12)\leq (\frac12+\frac{\upsilon}{4})(p^2+q^4+1).
\end{align*}
\begin{proposition}
\label{thm:2}
For each $\tau\in(0,1),$ the numerical solution $\{(P_n,Q_n)\}_{n\in \mathbb{N}}$ admits a unique invariant measure $\pi_{\tau}$. Furthermore, for $p\ge 1,$  there exist constants  $\kappa:=\kappa(p)\in (0,1)$ and
$C:=C(p)>0$ such that for any measurable functions $g$ satisfying  $|g|\leq (H+C_H)^p$, it holds 
\begin{align*}
|\mathbb{E}[g(P_n,Q_n)]-\pi_{\tau}(g)|\leq C\kappa^n(1+H^p(P_0,Q_0)).
\end{align*}
\end{proposition}
\begin{proof}
We first show the ergodicity for the sequence $\{(P_{2n},Q_{2n})\}_{n\in\mathbb N}.$  By \cite[Theorem 2.5]{Matting2002}, the proof is split into the following two steps, which are related to the  Lyapunov condition and the   minorization condition, respectively. 

\textit{Step 1.} 
We show that 
there exist constants $\alpha\in(0,1)$ and $\beta\in(0,\infty)$ such that $$
\mathbb{E}[H(P_{n+1},Q_{n+1})+C_H|\mathcal{F}_n]\leq \alpha (H(P_n,Q_n)+C_H)+\beta.$$ In fact, it follows from   \eqref{SAVF} that \begin{align*}
  &H(P_{n+1},Q_{n+1})
  =\frac{|P_{n+1}|^2}{2}+\frac{|Q_{n+1}|^4}{4}+\frac{\upsilon}{2} P_{n+1}Q_{n+1}
  \\
   = &\, e^{-\upsilon \tau}\Big(\frac{1}{2}|\bar{P}_{T_n,t_{n+1}}|^2+\frac{1}{4}e^{-\upsilon\tau}|\bar{Q}_{T_n,t_{n+1}}|^4+\frac{\upsilon}{2}\bar{P}_{T_n,t_{n+1}}\bar{Q}_{T_n,t_{n+1}}\\
   &+\frac12|\int_{t_n}^{t_{n+1}}e^{-\frac{\upsilon}{2}(t_{n+1}-t)}\sigma
  dW_t|^2e^{\upsilon \tau}
+e^{\frac{\upsilon\tau}{2}}\bar{P}_{T_n,t_{n+1}}\int_{t_n}^{t_{n+1}}e^{-\frac{\upsilon}{2}(t_{n+1}-t)}\sigma dW_t\\&+\frac{\upsilon}{2}e^{\frac{\upsilon\tau}{2}}\bar{Q}_{T_n,t_{n+1}}\int_{t_n}^{t_{n+1}}e^{-\frac{\upsilon}{2}(t_{n+1}-t)}\sigma dW_t\Big).
\end{align*}
Noting  that $(\bar{P}_{T_n,t_{n+1}},\bar{Q}_{T_n,t_{n+1}})^{\top}$ is $\mathcal F_{t_n}$-measurable, which is independent of the increment $\int_{t_n}^{t_{n+1}}e^{-\frac{\upsilon}{2}(t_{n+1}-t)}\sigma dW_t$, we obtain   
\begin{align*}  \mathbb{E} [H(P_{n+1},Q_{n+1})|\mathcal F_{t_n}]
&\leq e^{-\upsilon \tau}H(\bar{P}_{T_{n},t_{n+1}},\bar{Q}_{T_n,t_{n+1}})+\frac{\sigma^2(1-e^{-\upsilon \tau})}{2\upsilon}\\
&
=e^{-\upsilon\tau}H(P_n,Q_n)+\frac{\sigma^2(1-e^{-\upsilon \tau})}{2\upsilon},
\end{align*}
where in the last step we use that the subsystem 
 \eqref{eq:2-1} preserves the Hamiltonian  $H$. 
 
 \textit{Step $2$.} For any given $y,y^*\in\mathbb R^2, $ we show that $\Delta W_0$ and $\Delta W_1$ can be determined to ensure that $(P_2, Q_2)^{\top}=y*=(y^*_1,y^*_2)^{\top}$ for any initial value $(P_0,Q_0)^{\top}=y$.  
In fact, the solvability  of 
\eqref{eq:2-1}  gives that $\bar P_{T_0,t_1}=\psi_{\tau}(y)$ and $\bar Q_{T_0,t_1}=\phi_{\tau}(y)$, which together  with \eqref{sub_ex3} leads to \begin{align*}
&P_1=e^{-\frac{\upsilon \tau}{2}}\psi_{\tau}(y)+\int_{t_0}^{t_1}e^{-\frac{\upsilon}{2}(t_1-t)}\sigma dW_t\,{\overset{d}=}\,e^{-\frac{\upsilon \tau}{2}}\psi_{\tau}(y)+\sigma\sqrt{\frac{1-e^{-\upsilon \tau}}{\upsilon\tau}}\Delta W_0,\\
&Q_1=e^{-\frac{\upsilon \tau}{2}}\phi_{\tau}(y),
\end{align*} 
where $X \,{ \overset{d}= }\,Y$ means that random variables $X,Y$ are equal in distribution.  Using again the solvability of $\bar Q_{T_1,t_2}$, we have 
$
\bar Q_{T_1,t_2}=\phi_{\tau}(P_1,Q_1)=e^{\frac{\upsilon \tau}{2}}Q_2=e^{\frac{\upsilon \tau}{2}}y^*_2,
$ which determines $P_1$ and hence gives  $\Delta W_0$. Then utilizing $y^*_1=P_2\,{\overset{d}=}\,e^{-\frac{\upsilon \tau}{2}}\psi_{\tau}(P_1,Q_1)+\sigma\sqrt{\frac{1-e^{-\upsilon \tau}}{\upsilon\tau}}\Delta W_1$ determines $\Delta W_1.$ In addition, the above  procedure  implies that the transition kernel $\mathbb{P}_{2n}(x,A),\,x\in\mathbb R^2,A\in\mathcal B(\mathbb R^2)$ admits a continuous  density function $p_{2n}(x,y)$ satisfying $
  \mathbb{P}_{2n}(x,A)=\int_Ap_{2n}(x,y)dy.$ 
  
Combining \textit{Steps 1--2} gives that $|\mathbb E[g(P_{2n},Q_{2n})]-\pi_{\tau}(g)|\leq C\kappa^n(1+H^p(P_0,Q_0))$. Then similar to the proof of \cite[Theorem 7.3]{Matting2002} and in virtue of  Lemma \ref{lem:2}, one can obtain the desired result for the  sequence $\{(P_n,Q_n)\}_{n\in\mathbb N}.$ The proof is finished. 
\end{proof}

The following proposition shows the exponential integrability of the numerical solution of \eqref{SAVF}, which indicates the boundedness of exponential moment of the numerical solution. 
\begin{proposition}
  \label{lem:3.1}
 There exists a  constant $C>0$ such that   
\begin{equation*}
\sup_{t_n\in[0,T]}\mathbb E\Big[\exp\Big\{\frac{C_e(|P_n|^2+|Q_n|^4)}{e^{\sigma^2t_n}}\Big\}\Big]\leq e^{C(T+1)+H(P_0,Q_0)},  
\end{equation*}
where $C_e$ is given in \eqref{equ:1}.
\end{proposition}

\begin{proof}
  Denote $\tilde{\mu}(p,q):=-\frac{\upsilon}{2}(p,q)^{\top}$ and $\tilde{\sigma}:=(\sigma,0)^{\top}$. 
  Then  we have
\begin{align*}
&\mathcal{D}H(p,q)\tilde{\mu}(p,q)=-\frac{\upsilon}{2}p^2-\frac{\upsilon}{2}q^4-\frac{\upsilon^2}{2}pq,\\
&tr(\mathcal{D}^2H(p,q)\tilde{\sigma}\tilde{\sigma}^T)=\sigma^2,\quad |\tilde{\sigma}^T\mathcal{D}H(p,q)|^2=|\sigma(p+\frac{\upsilon}{2}q)|^2.
\end{align*}
 This implies 
\begin{align*}
 \tilde{\mathcal E} :=&\,\mathcal{D}H(\tilde{P}_{T_n,t},\tilde{Q}_{T_n,t})\tilde{\mu}(\tilde{P}_{T_n,t},\tilde{Q}_{T_n,t})+\frac{tr(\mathcal{D}^2H(\tilde{P}_{T_n,t},\tilde{Q}_{T_n,t})\tilde\sigma\tilde\sigma^T)}{2}+\frac{|\tilde\sigma^T\mathcal{D}H(\tilde{P}_{T_n,t},\tilde{Q}_{T_n,t})|^2}{2}
  \\
  =&\, -\frac{\upsilon}{2}\big(|\tilde{P}_{T_n,t}|^2+\upsilon \tilde{P}_{T_n,t}\tilde{Q}_{T_n,t}+|\tilde{Q}_{T_n,t}|^4\big)+\frac{\sigma^2}{2e^{\beta t}}\big(|\tilde{P}_{T_n,t}|^2+\upsilon \tilde{P}_{T_n,t}\tilde{Q}_{T_n,t}+\frac{\upsilon^4}{4}|\tilde{Q}_{T_n,t}|^2\big)\\
  &+\frac{\sigma^2}{2}\leq -\upsilon H(\tilde{P}_{T_n,t},\tilde{Q}_{T_n,t})+\sigma^2 H(\tilde{P}_{T_n,t},\tilde{Q}_{T_n,t})+\frac{\sigma^2\upsilon^4}{64}+\frac{\sigma^2}{2}.
 \end{align*}
 By \cite[Lemma 2.5]{dang2024}, we obtain 
  \begin{align*}
   & \mathbb{E}\Big[\exp\Big\{\frac{H(\tilde{P}_{T_n,t},\tilde{Q}_{T_n,t})}{e^{\sigma^2(t-t_n)}}+\int_{t_n}^t\frac{\upsilon H(\tilde{P}_{T_n,r},\tilde{Q}_{T_n,r})-\frac{\sigma^2}{2}-\frac{\sigma^2\upsilon^4}{64}}{e^{\sigma^2(r-t_n)}}dr\Big\}\Big]\\&\leq \mathbb E\big[\exp\big\{H(\tilde{P}_{T_n,t_n},\tilde{Q}_{T_n,t_n})\big\}\big]
    =\mathbb E\big[\exp\big\{H(P_n,Q_n)\big\}\big],
\end{align*}
where in the last equality we use the preservation of the Hamiltonian for \eqref{eq:2-1}. 
By considering $e^{-\sigma^2t_n}H$ instead of $H$ and using \eqref{HCH}, we arrive at 
\begin{align*}
 \mathbb{E}\Big[\exp\Big\{\frac{H(\tilde{P}_{T_n,t},\tilde{Q}_{T_n,t})}{e^{\sigma^2t}}\Big\}\Big]&\leq \exp\Big\{\big(\upsilon C_H+\frac{\sigma^2}{2}+\frac{\sigma^2\upsilon^4}{64}\big)(t-t_n)\Big\}\mathbb E\Big[\exp\Big\{\frac{H(P_n,Q_n)}{e^{\sigma^2t_n}}\Big\}\Big].
\end{align*}
By iteration and combining \eqref{equ:1}, we complete the proof. 
\end{proof}

Similarly, one can also prove the uniform moment boundedness, ergodicity, and the exponential integrability of  the solution for the  splitting system \eqref{split}. For our convenience, we list these properties in the following  proposition  and omit the proof.
\begin{proposition}
\label{lem:2-1}
 It holds that 

(\romannumeral1) for any $ p\geq 1$, there exists a  constant $C:=C(p)>0$ such that  
\begin{align*}
\sup_{n\in \mathbb{N}}\sup_{t\in T_n}\mathbb{E}[(\bar P_{T_n}(t))^{2p}+(\bar Q_{T_n}(t))^{4p}+(\tilde P_{T_n}(t))^{2p}+(\tilde Q_{T_n}(t))^{4p}]\leq C(1+H^p(P_0,Q_0)).
\end{align*}
In addition, the solution of the splitting system \eqref{split} is uniquely ergodic; 

(\romannumeral2) there exists a   constant $C>0$ such that 
\begin{align*}
 \mathbb E\Big[\exp\Big\{\frac{H(P_{\tau}(t_n),Q_{\tau}(t_n))}{e^{\sigma^2t_n}}\Big\}\Big]\leq \exp(H(P_0,Q_0)+Ct_n).
 \end{align*}
\end{proposition}

\section{Strong and weak convergence orders}\label{order} In this section, 
we first present that under  moderate verifiable conditions on the one-step approximation \eqref{eq:2-1}, the splitting method \eqref{SAVF} converges to the exact solution with the  strong convergence order of $1$. The proof is also given based on the decomposition of the error, and the utilization  of the exponential integrability and the stochastic Gr\"onwall inequality. 
Then based on the Kolmogorov equation and the uniform  moment boundedness, we  prove  the first-order weak convergence for approximating the ergodic limit via the numerical ergodic limit.
The main result on the strong convergence order is stated as follows.
\begin{theorem}
\label{thm:1}
Suppose that the one-step approximation  \eqref{eq:2-1} satisfies 
\begin{align}
&\big|\frac{\upsilon}{2}P_m+Q^3_m+\tau^{-1}\mathcal A^m_{\tau}\big|\leq C\big(|\bar P_{T_m,t_{m+1}}-P_m|^a+|\bar Q_{T_m,t_{m+1}}-Q_m|^b\big)\Theta_1^m,\label{cond1}\\
&\big|P_m+\frac{\upsilon}{2}Q_m-\tau^{-1}\mathcal B^m_{\tau}\big|\leq C\big(|\bar P_{T_m,t_{m+1}}-P_m|^a+|\bar Q_{T_m,t_{m+1}}-Q_m|^b\big)\Theta_2^m\label{cond2}
\end{align}
 for constants  $C>0,a\wedge b\ge 1,$ where maps $\Theta_1^m:=\Theta_1(P_m,Q_m,\bar P_{T_m,t_{m+1}},\bar Q_{T_m,t_{m+1}})$ and  $\Theta_2^m:=\Theta_2(P_m,Q_m,\bar P_{T_m,t_{m+1}},\bar Q_{T_m,t_{m+1}})$ are of polynomial growth. In addition let \begin{align}\label{cond3}\sup_{t_n\in[0,T]}(\|\tau^{-1}\mathcal A^n_{\tau}\|_{L^{6pa}(\Omega)}+\|\tau^{-1}\mathcal B^n_{\tau}\|_{L^{6pb}(\Omega)})\leq C
 \end{align} for $p>0.$ 
Then for $ T>0$, there exist $\tilde{\tau}_0:=\tilde{\tau}_0(T)\in(0,1)$ and  $C:=C(T)>0,$ such that for $\tau\in(0,\tilde{\tau}_0)$ and $p>0,$ 
\begin{align*}
\sup_{t_m\in[0,T]}\|(P(t_m),Q(t_m))^{\top}-(P_m,Q_m)^{\top}\|_{L^{2p}(\Omega)}\leq C\tau.
\end{align*}
\end{theorem}

Below we present the  convergence order  between numerical and exact ergodic limits.  
\begin{theorem}
\label{thm:3} Under conditions in Theorem \ref{thm:1},  
for  measurable functions 
$ g\in \mathcal{G}_l=\{g:|g(x)-g(y)|\leq C(1+|x|^{2l-1}+|y|^{2l-1})|x-y|,|g(x)|\leq (H(x)+C_H)^l\}$, 
we have 
$
|\pi(g)-\pi_{\tau}(g)|\leq C\tau,
$
where $C>0$ is time-independent. 
\end{theorem}
\begin{proof}
Introduce the function $u(t,p,q):=\mathbb E_{(p,q)}[g(P(t),Q(t))]-\pi(g)$, which is the solution of the Kolmogorov equation $\frac{\partial u(t,p,q)}{\partial t}=\mathcal Lu(t,p,q).$ Here, $\mathcal L$ is the infinitesimal generator of \eqref{eq:0} and $\mathbb E_{(p,q)}$ is the conditional expectation with respect to the initial $(p,q)$. 
We let 
$b_2(p,q)=(-\frac{\upsilon }{2}p,-\frac{\upsilon}{2}q)^{\top}.$ 
  It follows from \eqref{SAVF} and \eqref{eq:2-1} that 
\begin{align*}
X_{n+1}=X_n+(\mathcal A^n_{\tau},\mathcal B^n_{\tau})^{\top}+\int_{t_n}^{t_{n+1}}b_2(\tilde P_{T_n,r},\tilde Q_{T_n,r})dr+(\sigma\Delta W_n,0)^{\top}.
\end{align*} 
Then by the Taylor expansion, \eqref{eq:2-1}, and  \eqref{cond1}--\eqref{cond2}, one can derive that for $j\ge 0,n\ge 0,$
\begin{align}\label{inter1}
\mathbb E[u(j\tau,X_{n+1})]\leq \mathbb E[u(j\tau,X_n)]+\mathbb E[\mathcal L u(j\tau,X_n)]\tau+\mathcal C_{j,n+1}\tau+r_{j,n+1}\tau^2.
\end{align} 
Here, the function $\mathcal C_{j,n+1}$ is the   term of  type
$\mathbb E\big[\|Du(j\tau,X_n)\|\big(\tau \mathbb J^n_{1,1}+\int_{t_n}^{t_{n+1}}\mathbb J^n_{1,2}(r)dr\big)+\tau (\|D^2u(j\tau,X_n)\|+\|D^3u(j\tau,X_n+\theta(X_{n+1}-X_n))\|)\mathbb J_{1,3}^n\big],$  where $D^iu$ is the standard $i$th derivative of $u$,   $\theta\in(0,1),$ $\mathbb J^n_{1,1}:=\mathbb J_{1,1}(X_n,\bar P_{T_n,t_{n+1}},\bar Q_{T_n,t_{n+1}},\tau^{-1}\mathcal A^n_{\tau},\tau^{-1}\mathcal B^n_{\tau}),$ $\mathbb J^n_{1,2}(r):=\mathbb J_{1,2}(X_n,\tilde P_{T_n,r},\tilde Q_{T_n,r})$, and $\mathbb J^n_{1,3}:=\mathbb J_{1,3}(X_n)$ are  functions with polynomial growth at infinity. The remainder term $r_{j,n+1}$ is the  term of type
$\mathbb E\big[ (\|D^2u(j\tau,X_n)\|+\|D^3u(j\tau,X_n+\theta(X_{n+1}-X_n))\|)\big(\tau \mathbb J^n_{2,1}+\int_{t_n}^{t_{n+1}}\mathbb J^n_{2,2}(r)dr \big)\big],$ where functions $\mathbb J^n_{2,1}:=\mathbb J_{2,1}(X_n,\bar P_{T_n,t_{n+1}},\bar Q_{T_n,t_{n+1}},\tau^{-1}\mathcal A^n_{\tau},\tau^{-1}\mathcal B^n_{\tau}),$ and $\mathbb J_{2,2}^n(r):=\mathbb J_{2,2}(X_n,\tilde P_{T_n,r},\tilde Q_{T_n,r})$ are of  polynomial growth at infinity.  According to \cite[Theorem 3.1]{Talay2002}, Lemma \ref{lem:2}, and \eqref{cond3}, we obtain that for some $q_i>0,i=0,1,2$ and $C>0,$
\begin{align*}
&\quad \sup_{n\ge 0}\sum_{j=0}^{\infty}|r_{j,n+1}|\\
&\leq C\big(1+\sup_{n\in\mathbb N}(\|X_n\|^{q_1}_{L^{q_0}(\Omega)}+\|\mathbb J_{1,1}^n\|_{L^{q_2}(\Omega)}+\sup_{r\in T_n}\|\mathbb J_{1,2}^n(r)\|_{L^{q_2}(\Omega)})\big)\tau\sum_{j=0}^{\infty}e^{-Cj\tau}\leq C.
\end{align*}

It is observed that the Taylor expansion gives 
{\small\begin{align}\label{inter2}
\mathbb E[u((j+1)\tau,X_n)]=\mathbb E[u(j\tau,X_n)]+\mathbb E[\mathcal Lu(j\tau,X_n)]\tau+\frac12\mathbb E[\mathcal L^2u(j\tau,X_n)]\tau^2+\tilde r_{j+1,n}.
\end{align}} 
Here, the term $\mathcal L^2u(j\tau,X_n)$ is a summation of terms of type $\partial ^iu(j\tau,X_n) \tilde {\mathbb J}_{1}(X_n),i=1,2,3,4,$ where $\partial^i$ is the $i$th partial derivatives with respect to $(p,q)$ and $\tilde {\mathbb J}_{1}$ is a polynomial.  And the 
remainder $\tilde r_{j+1,n}$ is a summation of terms of type 
$\int_{j\tau}^{(j+1)\tau}\int_0^1\partial ^iu(j\tau+\tilde{\theta}(t-j\tau),X_n)\tilde {\mathbb J}_{2}(X_n)(t-j\tau)^2d\tilde{\theta}dt,i=1,\ldots,6,$ where $\tilde{\mathbb J}_2$ is a polynomial.  


Combining \eqref{inter1} and \eqref{inter2} leads to 
 \begin{align*}
\mathbb E[u(j\tau,X_{n+1})]\leq \mathbb E[u((j+1)\tau,X_n)]+\tilde {\mathcal C}_{j,n+1}\tau+R_{j+1,n+1},
 \end{align*}
 where $ \tilde {\mathcal C}_{j,n+1}$ is the   summation  of  terms $\mathcal C_{j,n+1}$ and 
$\mathbb E\big[\tau \|D^iu(j\tau,X_n)\|\cdot|\tilde{\mathbb J}_{1}(X_n)|\big],i=1,\ldots,4,$ and $ R_{j+1,n+1}$ is a summation of terms $r_{j,n+1}\tau^2$ and $\int_{j\tau}^{(j+1)\tau}\int_0^1\mathbb E[|\tilde {\mathbb J}_2(X_n)|\cdot\|D^iu(j\tau+\tilde{\theta}(t-j\tau))\|](t-j\tau)^2d\tilde{\theta}dt,i=1,\ldots,6.$ 
 Therefore, we arrive at
 \begin{align*}
 &\quad \frac1N\sum_{n=1}^N\mathbb E[g(X_n)-\pi(g)]=\frac1N\sum_{n=1}^N\mathbb E[u(0,X_n)]\\
 &\leq \frac1N\sum_{n=1}^N\mathbb E\Big[u(n\tau,X_0)+\sum_{j=0}^{n-1}\tilde{\mathcal C}_{j,n-j}\tau+\sum_{j=0}^{n-1}R_{j+1,n-j}\Big].
 \end{align*}
 The  ergodicity of the exact solution implies $\frac1N\sum_{n=1}^N\mathbb E[u(n\tau,X_0)]\to0$ as $N\to\infty.$  From \cite[Theorem 3.1]{Talay2002}, Lemma \ref{lem:2}, and \eqref{cond3}, we obtain 
$
 \sup\limits_{N\ge 1}\frac1N\sum_{n=1}^N\sum_{j=0}^{n-1}\tilde{\mathcal C}_{j,n-j}\leq C\sum_{j=0}^{\infty}e^{-Cj\tau}\tau\leq C 
$ and $\sup\limits_{N\ge 1}\frac1N\sum_{n=1}^N\sum_{j=0}^{n-1}R_{j+1,n-j}\leq C\tau^3\sum_{j=0}^{\infty}e^{-Cj\tau}\leq C\tau^2.$ 
 Hence, by virtue of the ergodicity of the numerical solution, we derive that 
$
 |\pi(g)-\pi_{\tau}(g)|\leq C\tau.
$ The proof is finished. 
  \end{proof}

 \begin{proof}[Proof of Theorem \ref{thm:1}] The proof is divided into two steps based on the decomposition of the error. 
 
 \textit{Step 1. 
Show that for $T>0$, there exist $\tilde{\tau}_0:=\tilde{\tau}_0(T)\in(0,1)$ and  $C:=C(T)>0,$ such that for any $\tau\in(0,\tilde{\tau}_0)$ and $p>0,$ it holds
$
\sup_{t_m\in[0,T]}\|(P(t_m),Q(t_m))^{\top}-(P_{\tau}(t_m),Q_{\tau}(t_m))^{\top}\|_{L^{2p}(\Omega)}\leq C\tau.$}

Let  $e(t_{m}):=P(t_{m})-{P}_{\tau}(t_{m})$ and $\tilde{e}(t_m):=Q(t_m)-Q_{\tau}(t_m)$.
By \eqref{eq:0} and \eqref{split},  we have  
\small{
\begin{align}\label{equ1}
&P(t)-\tilde{P}_{T_m}(t)\notag\\
  =&\, e(t_m)+\int_{t_m}^{t_{m+1}}\big(\frac{\upsilon}{2}\bar{P}_{T_m}(s)+\bar{Q}_{T_m}^3(s)\big)ds+\int_{t_m}^t\big(\frac{\upsilon}{2}\tilde{P}_{T_m}(s)-\upsilon P(s)-Q^3(s)\big)ds. 
\end{align}}
This, together with the Taylor formula yields 
  \begin{align}\label{rela_e}
    |e(t_{m+1})|^2=&\,\Big|e(t_m)+\int_{t_m}^{t_{m+1}}\big(\frac{\upsilon}{2}\bar{P}_{T_m}(s)+\bar{Q}_{T_m}^3(s)\big)ds\Big|^2+\int_{t_m}^{t_{m+1}}(P(t)-\tilde{P}_{T_m}(t))\times\notag\\
    &\big(\upsilon \tilde{P}_{T_m}(t)-2\upsilon P(t)-2Q^3(t)\big)dt 
  =\, |e(t_m)|^2+J_m^1+J_m^2+J_m^3+J_m^{4},
 \end{align}
 where 
 \begin{align*}
 &J_m^1:=\int_{t_m}^{t_{m+1}}e(t_m)\big(\upsilon \bar{P}_{T_m}(s)+2\bar{Q}_{T_m}^3(s)+\upsilon \tilde{P}_{T_m}(s)-2\upsilon P(s)-2Q^3(s)\big)ds,\\
 &J_m^2:=\Big|\int_{t_m}^{t_{m+1}}\big(\frac{\upsilon}{2}\bar{P}_{T_m}(s)+\bar{Q}_{T_m}^3(s)\big)ds\Big|^2,\\
 &J^3_m:=\int_{t_m}^{t_{m+1}}\int_{t_m}^{t_{m+1}}\big(\frac{\upsilon}{2}\bar{P}_{T_m}(s)+\bar{Q}_{T_m}^3(s)\big)\big(\upsilon \tilde{P}_{T_m}(t)-2\upsilon
  P(t)-2Q^3(t)\big)dsdt,\\
  &J^4_m:=\int_{t_m}^{t_{m+1}}\int_{t_m}^t\big(\frac{\upsilon}{2}\tilde{P}_{T_m}(s)-\upsilon P(s)-Q^3(s)\big)\big(\upsilon \tilde{P}_{T_m}(t)-2\upsilon P(t)-2Q^3(t)\big)dsdt.
 \end{align*}
 For the term $J^1_m,$ it can be decomposed as  
\begin{align*}
    J_m^1=&\int_{t_m}^{t_{m+1}}\upsilon e(t_m)(\bar{P}_{T_m}(s)-P(s))ds+\int_{t_m}^{t_{m+1}}\upsilon e(t_m)(\tilde{P}_{T_m}(s)-P(s))ds\\
    &+2\int_{t_m}^{t_{m+1}}e(t_m)(\bar{Q}_{T_m}^3(s)-Q^3(s))ds  =:J_m^{1,1}+J_m^{1,2}+J_m^{1,3}.
\end{align*} 
Note that 
 \begin{align}
&\bar{P}_{T_m}(s)-P(s)=-e(t_m)+\int_{t_m}^s\big(\upsilon P(r)-\frac{\upsilon}{2} \bar{P}_{T_m}(r)-\bar{Q}_{T_m}^3(r)+Q^3(r)\big)dr-\sigma (W_s-W_{t_m}),\label{equP1}
    \\
&\bar{Q}_{T_m}(s)-Q(s)=-\tilde{e}(t_m)+\int_{t_m}^s\big(\bar{P}_{T_m}(r)-P(r)+\frac{\upsilon}{2}\bar{Q}_{T_m}(r)\big)dr.\label{equP2}
\end{align}
Combining \eqref{equ1} leads to that 
  \begin{align*}
J_m^{1,1}=&-\upsilon |e(t_m)|^2\tau+\upsilon\int_{t_m}^{t_{m+1}}\int_{t_m}^s e(t_m)\big(\upsilon
P(r)-\frac{\upsilon}{2}\bar{P}_{T_m}(r)\big)drds\\
&+\upsilon\int_{t_m}^{t_{m+1}}\int_{t_m}^se(t_m)\big(-\bar{Q}_{T_m}^3(r)+Q^3(r)\big)drds
-\upsilon\int_{t_m}^{t_{m+1}}e(t_m)\sigma (W_s-W_{t_m})ds,
\\
J_m^{1,2}=&-\upsilon
|e(t_m)|^2\tau+\upsilon\int_{t_m}^{t_{m+1}}\int_{t_m}^{t_{m+1}}e(t_m)\big(-\frac{\upsilon}{2}\bar{P}_{T_m}(r)-\bar{Q}_{T_m}^3(r)\big)drds\\
&+\upsilon\int_{t_m}^{t_{m+1}}\int_{t_m}^se(t_m)\big(-\frac{\upsilon}{2}\tilde{P}_{T_m}(r)+\upsilon P(r)+Q^3(r)\big)drds,
\end{align*}
and
\begin{align*}
  J_m^{1,3}=&\,2\int_{t_m}^{t_{m+1}}e(t_m)(\bar{Q}_{T_m}(s)-Q(s))\big(\bar{Q}_{T_m}^2(s)+\bar{Q}_{T_m}(s)Q(s)+Q^2(s)\big)ds
  \\
  =&-2e(t_m)\tilde{e}(t_m)\int_{t_m}^{t_{m+1}}\big(\bar{Q}_{T_m}^2(s)+\bar{Q}_{T_m}(s)Q(s)+Q^2(s)\big)ds
  \\
  &+2\int_{t_m}^{t_{m+1}}\int_{t_m}^se(t_m)\big(\bar{P}_{T_m}(r)-P(r)+\frac{\upsilon}{2}\bar{Q}_{T_m}(r)\big)\big(\bar{Q}_{T_m}^2(s)+\bar{Q}_{T_m}(s)Q(s)+Q^2(s)\big)drds.
\end{align*}
Hence, by the H\"older inequality and the Young inequality, we obtain 
\begin{align}\label{inequPJ1}
J^1_m&\leq -\frac{\upsilon\tau}{2}|e(t_m)|^2+C(|e(t_m)|^2+|\tilde e(t_m)|^2)\int_{t_m}^{t_{m+1}}(|Q(s)|^2+|\bar Q_{T_m}(s)|^2)ds\notag\\
&\quad+
C\tau^2\int_{t_m}^{t_{m+1}}\big(|P(s)|^6+|\bar P_{T_m}(s)|^6+|\tilde P_{T_m}(s)|^2+|Q(s)|^6+|\bar Q_{T_m}(s)|^6+1\big)ds\notag\\
&\quad -\upsilon\int_{t_m}^{t_{m+1}}e(t_m)\sigma (W_s-W_{t_m})ds.
\end{align} 
For terms $J_m^2, J_m^3, J_m^4,$
it follows from   $(a+b)^2=a^2+2ab+b^2,a,b\in\mathbb R$ that 
\begin{align*}
  &\quad J_m^2+J_m^3+J_m^4
    \\
&=\Big|\frac{\upsilon}{2}\int_{t_m}^{t_{m+1}}(\bar{P}_{T_m}(t)-P(t))dt+\frac{\upsilon}{2}\int_{t_m}^{t_{m+1}}(\tilde{P}_{T_m}(t)-P(t))dt+\int_{t_m}^{t_{m+1}}(\bar{Q}_{T_m}^3(t)-Q^3(t))dt\Big|^2
\\
&\quad -\int_{t_m}^{t_{m+1}}\int_t^{t_{m+1}}\big(\frac{\upsilon}{2}\tilde{P}_{T_m}(s)-\upsilon
P(s)-Q^3(s)\big)\big(\frac{\upsilon}{2}\tilde{P}_{T_m}(t)-\upsilon P(t)-Q^3(t)\big)dsdt
\\
&\quad +\int_{t_m}^{t_{m+1}}\int_{t_m}^t\big(\frac{\upsilon}{2}\tilde{P}_{T_m}(s)-\upsilon P(s)-Q^3(s)\big)\big(\frac{\upsilon}{2}\tilde{P}_{T_m}(t)-\upsilon P(t)-Q^3(t)\big)dsdt
\\
&=\Big|\frac{\upsilon}{2}\int_{t_m}^{t_{m+1}}(\bar{P}_{T_m}(t)-P(t))dt+\frac{\upsilon}{2}\int_{t_m}^{t_{m+1}}(\tilde{P}_{T_m}(t)-P(t))dt+\int_{t_m}^{t_{m+1}}(\bar{Q}_{T_m}^3(t)-Q^3(t))dt\Big|^2,
\end{align*}
where the second equality uses the integral transformation.  Using \eqref{equ1}, \eqref{equP1},  \eqref{equP2}, the H\"older inequality, and the Young inequality,  
we derive that 
\begin{align}\label{inequP}
J^2_m+J^3_m+J^4_m&\leq C\Big[\tau^2|e(t_m)|^2+C\tau|\tilde e(t_m)|^2\int_{t_m}^{t_{m+1}}(|\bar Q_{T_m}(s)|^4+|Q(s)|^4)ds+\tau^3\int_{t_m}^{t_{m+1}}\big(|P(s)|^6\notag\\
& +|\tilde P_{T_m}(s)|^2+|\bar P_{T_m}(s)|^6+|\bar Q_{T_m}(s)|^6+|Q(s)|^6\big)ds+\tau \int_{t_m}^{t_{m+1}}|\sigma (W_t-W_{t_m})|^2dt\Big].
\end{align}

By the Taylor formula and 
\begin{align*}
Q(t)- \tilde{Q}_{T_m}(t)
  =\tilde{e}(t_m)-\int_{t_m}^{t_{m+1}}\big(\bar P_{T_m}(s)+\frac{\upsilon}{2}\bar{Q}_{T_m}(s)\big)ds+\int_{t_m}^t\big(P(s)+\frac{\upsilon}{2}\tilde{Q}_{T_m}(s)\big)ds, 
\end{align*}
we derive 
  \begin{align}\label{tilde_e}
    |\tilde{e}(t_{m+1})|^2=&\,\Big|\tilde{e}(t_m)-\int_{t_m}^{t_{m+1}}\big(\bar{P}_{T_m}(s)+\frac{\upsilon}{2}\bar{Q}_{T_m}(s)\big)ds\Big|^2+\int_{t_m}^{t_{m+1}}(Q(t)-\tilde{Q}_{T_m}(t))\times \notag\\
    &(2P(t)+\upsilon \tilde{Q}_{T_m}(t))dt=\, |\tilde{e}(t_m)|^2+K_m^1+K_m^2+K_m^3+K_m^4,
\end{align}
where 
\begin{align*}
&K_m^1:=2\int_{t_m}^{t_{m+1}}\tilde{e}(t_m)(P(t)-\bar{P}_{T_m}(t))dt+\upsilon\int_{t_m}^{t_{m+1}}\tilde{e}(t_m)(\tilde{Q}_{T_m}(t)-\bar{Q}_{T_m}(t))dt,\\
&K_m^2:=\Big|\int_{t_m}^{t_{m+1}}\big(\bar{P}_{T_m}(s)+\frac{\upsilon}{2}\bar{Q}_{T_m}(s)\big)ds\Big|^2,\\
&K_m^3:=-\int_{t_m}^{t_{m+1}}\int_{t_m}^{t_{m+1}}\big(\bar{P}_{T_m}(s)+\frac{\upsilon}{2}\bar{Q}_{T_m}(s)\big)(2P(t)+\upsilon \tilde{Q}_{T_m}(t))dsdt,\\
&K_m^4:=\int_{t_m}^{t_{m+1}}\int_{t_m}^t\big(P(s)+\frac{\upsilon}{2}\tilde{Q}_{T_m}(s)\big)\big(2P(t)+\upsilon \tilde{Q}_{T_m}(t)\big)dsdt.
\end{align*}
For the term $K_m^1$, it follows from \eqref{equP1},  
\begin{align}\label{equQ1}
\tilde Q_{T_m}(t)-\bar Q_{T_m}(t)=-\int_{t_m}^t\frac{\upsilon}{2}\tilde Q_{T_m}(r)dr+\int_t^{t_{m+1}}\big(\bar P_{T_m}(r)+\frac{\upsilon}{2}\bar Q_{T_m}(r)\big)dr,
\end{align}
and the H\"older inequality that 
\begin{align}\label{eq:2-5}
  K_m^1 
 &\leq   C\tau(|e(t_m)|^2+|\tilde{e}(t_m)|^2)+2\int_{t_m}^{t_{m+1}}\tilde{e}(t_m)\sigma(W_t-W_{t_m})dt\notag\\
 &\quad +C\tau^2\int_{t_m}^{t_{m+1}}(|P(s)|^2+|\bar P_{T_m}(s)|^2+|Q(s)|^6+|\bar Q_{T_m}(s)|^6+|\tilde Q_{T_m}(s)|^2+1)ds.  
\end{align}
For terms $K_m^i,i=2,3,4$, by \eqref{equP1}, \eqref{equQ1}, and the integral transformation, we have 
\begin{align}\label{inequQ1}
&K_m^2+K_m^3+K_m^4
=\,\Big|\int_{t_m}^{t_{m+1}}(\bar{P}_{T_m}(s)-P(s))ds+\frac{\upsilon}{2}\int_{t_m}^{t_{m+1}}(\bar{Q}_{T_m}(s)-\tilde{Q}_{T_m}(s))ds\Big|^2\notag\\
\leq&\,  C\tau^2|e(t_m)|^2+C\tau^3\int_{t_m}^{t_{m+1}}\big(|\tilde {P}_{T_m}(s)|^2+|P(s)|^2+|\bar{P}_{T_m}(s)|^2+|Q(s)|^6+|\bar{Q}_{T_m}(s)|^6\big)ds
\notag\\
&
+C\tau\int_{t_m}^{t_{m+1}}|\sigma (W_s-W_{t_m})|^2ds. 
\end{align}

Hence, adding  \eqref{rela_e} and \eqref{tilde_e}, and combining \eqref{inequPJ1}, \eqref{inequP}, \eqref{eq:2-5}, \eqref{inequQ1},  and the Young inequality, 
 we arrive at that for $\epsilon\in(0,1),$
\begin{align*}
  &|e(t_{m+1})|^2+|\tilde{e}(t_{m+1})|^2 \leq (|e(t_m)|^2+|\tilde{e}(t_m)|^2) \Big(1+C\int_{t_m}^{t_{m+1}}\big(1+(\tau+\epsilon)\times\\
  &(|\bar{Q}_{T_m}(s)|^4+|Q(s)|^4)\big)ds\Big) + C\tau^2\int_{t_m}^{t_{m+1}}\Gamma_m(s)    ds+C\tau\int_{t_m}^{t_{m+1}}|W_s-W_{t_m}|^2ds\\
    & +\sigma(2\tilde{e}(t_{m})-\upsilon e(t_m)) \int_{t_m}^{t_{m+1}}(t_{m+1}-s)dW_s,
  \end{align*}
  where $\Gamma_m(s):=\Gamma(P(s),Q(s),\bar{P}_{T_m}(s),\bar{Q}_{T_m}(s),\tilde{P}_{T_m}(s),\tilde{Q}_{T_m}(s)),s\in T_m$ is a polynomial with order no larger than six, and the Fubini theorem is also used to obtain $$M_m:=\sigma(2\tilde{e}(t_{m})-\upsilon e(t_m))\int_{t_m}^{t_{m+1}}(W_s-W_{t_m})ds=\sigma(2\tilde{e}(t_{m})-\upsilon e(t_m))\int_{t_m}^{t_{m+1}}(t_{m+1}-s)dW_s.$$ It follows that  $\{M_j\}_{j\in\mathbb N}$ is a martingale with $M_0=0.$ 
By iteration, we have
\begin{align*}
 &  |e(t_{m+1})|^2+|\tilde{e}(t_{m+1})|^2\leq
C\sum\limits_{j=0}^m(|e(t_j)|^2+|\tilde{e}(t_j)|^2)\int_{t_j}^{t_{j+1}}\Big(1+(\tau+\epsilon)(|\bar{Q}_{T_j}(s)|^4\\
&+|Q(s)|^4)\Big)ds
+C\tau^2\sum\limits_{j=0}^m\int_{t_j}^{t_{j+1}}\Gamma_j(s)ds+C\tau\sum\limits_{j=0}^m\int_{t_j}^{t_{j+1}}|W_s-W_{t_j}|^2ds +\sum\limits_{j=0}^mM_j.
\end{align*}
For $t\in T_j,$ denote $\lfloor t\rfloor:=t_j$, $n_{t}:=j$, and $\mathcal Y_{\lfloor t\rfloor}:=|e(t_j)|^2+|\tilde e(t_j)|^2$. Then we have \begin{align*}
\mathcal Y_{\lfloor t\rfloor}&\leq \int_0^{\lfloor t\rfloor}\Big(C\mathcal Y_{\lfloor s\rfloor}(1+(\tau+\epsilon)(|\bar Q_{T_{n_s}}(s)|^4+|Q(s)|^4))+C\tau^2\Gamma_{n_{s}}(s) +C\tau |W_s-W_{\lfloor s\rfloor}|^2\Big)ds\\
&\quad+\int_0^{\lfloor t\rfloor}\sigma(2\tilde e(\lfloor s\rfloor)-\upsilon e(\lfloor s\rfloor))(\lfloor s\rfloor+\tau-s)dW_s=:\int_0^{\lfloor t\rfloor}a_sds+\int_0^{\lfloor t\rfloor}b_sdW_s.
\end{align*}
 To obtain the desired strong convergence order, we need to utilize the   stochastic Gr\"onwall  inequality (see \cite[Corollary 2.5]{SGI}). Now we verify the condition $(34)$ in \cite[Corollary 2.5]{SGI} as follows: for $p_0\ge 2,$
 \begin{align*}
 &\<\mathcal Y_{\lfloor t\rfloor},a_t\>+\frac12|b_t|^2+\frac{p_0-2}{2}\frac{|\<\mathcal Y_{\lfloor t\rfloor },b_t\>|^2}{|\mathcal Y_{\lfloor t\rfloor }|^2}\leq C(1+(\tau+\epsilon)(|\bar Q_{T_{n_t}}(t)|^4+|Q(t)|^4))|\mathcal Y_{\lfloor t\rfloor}|^2\\&+(C\tau^2\Gamma_{n_{t}}(t)+C\tau|W_t-W_{\lfloor t\rfloor}|^2+(\lfloor t\rfloor+\tau-t)^2)^2.
 \end{align*}
 Thus   combining Lemmas  \ref{lem:2-1} and \cite[Lemma 2.3, Corollary 2.2]{Talay2002}, we have 
that for $p>0$ and some $p_0,p_1>0,$
\begin{align*}
\|\mathcal Y_{t_{m+1}}\|_{L^p(\Omega)}&\leq \Big\|\exp\Big\{C\int_0^{T}(1+(\tau+\epsilon)(|\bar Q_{T_{n_s}}(s)|^4+|Q(s)|^4))ds\Big\}\Big\|_{L^{p_1}(\Omega)}\times\\
&\quad \Big(\int_0^{T}\Big\|\frac{C\tau^2\Gamma_{n_{s}}(s)+C\tau|W_s-W_{\lfloor s\rfloor}|^2+(\lfloor s\rfloor+\tau-s)^2}{\exp\{\int_0^s(1+(\tau+\epsilon)(|\bar Q_{T_{n_u}}(u)|^4+|Q(u)|^4))du\}}\Big\|^2_{L^{p_0}(\Omega)}ds\Big)^{\frac12}.
\end{align*}
By the convexity  of the exponential  function, we have
  \begin{align}\label{expo1}
    &\mathbb{E}\Big[\exp\Big\{Cp_1 \int_0^T\big(1+(\tau+\epsilon)(|\bar{Q}_{T_{n_s}}(s)|^4+|Q(s)|^4)\big)ds\Big\}\Big]
    \notag\\
    \leq&\, \frac{1}{T}\int_0^T\mathbb{E}\Big[\exp\Big(TCp_1\big(1+(\tau+\epsilon)(1+|\bar{Q}_{T_{n_s}}(s)|^4+|Q(s)|^4)\big)\Big)\Big]ds=:\mathcal I_0. 
\end{align}
 Based on \eqref{equ:1} and Lemma  \ref{lem:2-1}, there exist positive constants $\tilde{\tau}_0:=\tilde{\tau}_0(T)$ and $\tilde{\epsilon}_0:=\tilde{\epsilon}_0(T),$ such that 
 $CTp_1(\tilde\tau_0+\tilde\epsilon_0)\leq C_ee^{-\sigma^2T},$ where $C_e$ is given in \eqref{equ:1}.   Then 
for   $\tau\in (0,\tilde {\tau}_0)$ and $\epsilon\in(0,\tilde{\epsilon}_0)$, we have    $\mathcal I_0\leq C.$  This finishes the proof of \textit{Step 1}.   

\textit{Step 2. Show that 
 for any $ T>0$, there exist $\tilde\tau_0:=\tilde\tau_0(T)\in(0,1)$ and $C:=C(T)>0$ such that for any $\tau\in(0,\tilde\tau_0)$ and $p>0,$ it holds 
$
\sup\limits_{t_m\in[0,T]}\|(P_{\tau}(t_m),Q_{\tau}(t_m))^{\top}-(P_m,Q_m)^{\top}\|_{L^{2p}(\Omega)}\leq C\tau.
$}

Denote $e_1(t_m):=P_{\tau}(t_m)-P_m$ and $\tilde e_1(t_m):=Q_{\tau}(t_m)-Q_m.$ 
By \eqref{sub_ex3} and   \eqref{eq:2-1},   we have that for $t\in T_m,$
\begin{align*}
  &\quad \tilde P_{T_m}(t)-\tilde P_{T_m,t}\\&=e_1(t_m)-\int_{t_m}^{t_{m+1}} \big(\frac{\upsilon}{2}\bar{P}_{T_m}(s)+\bar{Q}_{T_m}^3(s)+\tau^{-1}\mathcal A^m_{\tau}\big)ds
 -\frac{\upsilon}{2}\int_{t_m}^t(\tilde{P}_{T_m}(s)-\tilde P_{T_m,s})ds.
\end{align*}
It follows from  the Taylor  formula that $
|\tilde P_{T_m}(t)-\tilde P_{T_m,t}|^2=|e_1(t_m)|^2+I_{m,1}+I_{m,2}+I_{m,3}
-\upsilon\int_{t_m}^t(\tilde P_{T_m}(s)-\tilde P_{T_m,s})^2ds,\;t\in T_m,
$
where $I_{m,1} :=e_1(t_m)\int_{t_m}^{t_{m+1}}-2\big(\bar{Q}_{T_m}^3(s)-Q_m^3\big) ds,$ and 
  \begin{align*}
I_{m,2}&:= e_1(t_m)\int_{t_m}^{t_{m+1}}\Big(-\upsilon(\bar{P}_{T_m}(s)-P_m)-\upsilon P_m
  -2Q_m^3-2\tau^{-1}\mathcal A^m_{\tau}\Big)ds,
  \\
  I_{m,3} &:= \Big|\int_{t_m}^{t_{m+1}}\Big(\frac{\upsilon}{2}(\bar{P}_{T_m}(s)-P_m)+\bar Q_{T_m}^3(s)-Q_m^3+\frac{\upsilon}{2}\bar P_m+Q^3_m+\tau^{-1}\mathcal A^m_{\tau}\Big)
  ds \Big|^2.
\end{align*}
From \eqref{sub_ex1}, we have
\begin{align}
&\bar{P}_{T_m}(s)-P_m
=-\frac{\upsilon}{2}\int_{t_m}^s\bar{P}_{T_m}(r)dr-\int_{t_m}^s\bar{Q}_{T_m}^3(r)dr+e_1(t_m),\label{eqbarP1}\\
  &\bar{Q}_{T_m}(s)-Q_m 
  =\int_{t_m}^s\bar{P}_{T_m}(r)dr+\frac{\upsilon}{2}\int_{t_m}^s\bar{Q}_{T_m}(r)dr+\tilde e_1(t_m). \label{eqbarQ1}
\end{align} 
This, together with \eqref{eq:2-1}, \eqref{cond1}, \eqref{eqbarP1}, \eqref{eqbarQ1}, and the H\"older inequality,  and the Young inequality  yields 
\begin{align*}
&I_{m,1}+I_{m,2}+I_{m,3}
   \leq C(|e_1(t_m)|^2+|\tilde e_1(t_m)|^2 )\int_{t_m}^{t_{m+1}}(1+|\bar{Q}_{T_m}(s)|^2+|Q_m|^2)ds\\
   &+C\tau^2\int_{t_m}^{t_{m+1}}\Big(1+|\bar{P}_{T_m}(r)|^6
  +|\bar{Q}_{T_m}(r)|^6+|Q_m|^6+\big(|\tau^{-1}\mathcal A^m_{\tau}|^{2a}+|\tau^{-1}\mathcal B^m_{\tau}|^{2b}\big)|\Theta_1^m|^2\Big)dr. 
\end{align*}

From \eqref{sub_ex1}  and \eqref{sub_ex3}, we obtain
  \begin{align*}
    & \tilde Q_{T_m}(t)-\tilde Q_{T_m,t}
    \\
    =&\,\tilde e_1(t_m)+\int_{t_m}^{t_{m+1}}\big(\bar{P}_{T_m}(s)+ \frac{\upsilon}{2}\bar{Q}_{T_m}(s)-\tau^{-1}\mathcal B^m_{\tau} \big)ds  -\frac{\upsilon}{2}\int_{t_m}^t(\tilde{Q}_{T_m}(s)-\tilde Q_{T_m,s})ds. 
\end{align*}
By virtue of the Taylor formula gives   $
     |\tilde Q_{T_m}(t)-\tilde Q_{T_m,t}|^2
 =|\tilde e_1(t_m)|^2+\hat{I}_{m,1}+\hat{I}_{m,2}+\hat{I}_{m,3}-\upsilon\int_{t_m}^t(\tilde Q_{T_m}(s)-\tilde Q_{T_m,s})^2ds,
$
where $\hat{I}_{m,1}:=\upsilon \tilde e_1(t_m)\int_{t_m}^{t_{m+1}}\big(\bar{Q}_{T_m}(s)-Q_m\big)ds,$ and 
\begin{align*}
  &\hat{I}_{m,2}:=
  2\tilde e_1(t_m)\int_{t_m}^{t_{m+1}}\big(\bar{P}_{T_m}(s)-P_m+P_m+\frac{\upsilon}{2}Q_m-\tau^{-1}\mathcal B^m_{\tau}\big)ds,\\
   &
\hat{I}_{m,3}:=\Big|\int_{t_m}^{t_{m+1}}\Big(\bar{P}_{T_m}(s)-P_m+\frac{\upsilon}{2}(\bar Q_{T_m}(s)-Q_m)+P_m+\frac{\upsilon}{2}Q_m-\tau^{-1}\mathcal B^m_{\tau}\Big)ds\Big|^2.
\end{align*}
Then combining  \eqref{cond2}, \eqref{eqbarP1},  and \eqref{eqbarQ1}, we have 
\begin{align*}
&\hat{I}_{m,1}+\hat{I}_{m,2}+\hat{I}_{m,3}\leq C\tau (|e_1(t_m)|^2+|\tilde e_1(t_m)|^2)\\&+ 
C\tau^2 \int_{t_m}^{t_{m+1}}\Big(1+|\bar{P}_{T_m}(r)|^2+|\bar{Q}_{T_m}(r)|^6+(|\tau^{-1}\mathcal A^m_{\tau}|^{2a}+|\tau^{-1}\mathcal B^m_{\tau}|^{2b})|\Theta_2^m|^2\Big)dr.
\end{align*}
Hence, utilizing the Young inequality, we arrive at that for $\epsilon\in(0,1), $
\begin{align*}
&|e_1(t_{m+1})|^2+|\tilde e_1(t_{m+1})|^2\leq (|e_1(t_{m})|^2+|\tilde e_1(t_{m})|^2)\times\\
&\Big(1+C\int_{t_m}^{t_{m+1}}\big(C(\epsilon)+(\epsilon+\tau)(|\bar Q_{T_m}(s)|^4 +|Q_m|^4)\big)ds\Big)+C\tau^2\int_{t_m}^{t_{m+1}}\tilde\Gamma_m(s)ds,
\end{align*}
where $\tilde \Gamma_m(s):=\tilde\Gamma(P_m,Q_m,\bar P_{T_m}(s),\bar Q_{T_m}(s),\tau^{-1}\mathcal A^m_{\tau},\tau^{-1}\mathcal B^m_{\tau}),s\in T_m$ is a polynomial. 
Then by iteration, we have
\begin{align*}
  &\quad |e_1(t_{m+1})|^2+|\tilde e_1(t_{m+1})|^2
  \leq C\sum\limits_{j=0}^m(|e_1(t_j)|^2+|\tilde e_1(t_j)|^2)\times\\
  &\int_{t_j}^{t_{j+1}}\big(1+(\tau+\epsilon)(|\bar{Q}_{T_j}(s)|^4 +|Q_j|^4)\big)ds+C\tau^2\sum\limits_{j=0}^m\int_{t_j}^{t_{j+1}}\tilde \Gamma_j(s)ds.
\end{align*}
By the discrete Gr\"onwall inequality and taking $p$th moment,  we obtain 
\begin{align*}
  &\||e_1(t_{m+1})|^2+|\tilde e_1(t_{m+1})|^2\|_{L^p(\Omega)}
  \leq C\tau^2\Big\|\int_0^T\tilde\Gamma_{n_s}(s)ds\Big\|_{L^{2p}(\Omega)}\times\\
  &\Big\|\exp\Big\{C\int_0^T \big(1+(\epsilon+\tau)(|\bar{Q}_{T_{n_s}}(s)|^4+|Q_{n_s}|^4)ds\big)\Big\}\Big\|_{L^{2p}(\Omega)}. 
\end{align*}
By Lemma \ref{lem:2-1} (\romannumeral2) and the preservation of the Hamiltonian for \eqref{sub_ex1}, we also have 
\begin{align*}\mathbb E\Big[\exp\Big\{\frac{H(\bar P_{T_j}(s),\bar Q_{T_j}(s))}{e^{\sigma^2 t_j}}\Big\}\Big]=\mathbb E\Big[\exp\Big\{\frac{H(P_{\tau}(t_j),Q_{\tau}(t_j))}{e^{\sigma^2 t_j}}\Big\}\Big]\leq \exp\{H(P_0,Q_0)+Ct_j\}.\end{align*}
Then similar to the proof of \eqref{expo1}, using the convexity of the exponential function and  Lemma  \ref{lem:3.1}, there exist positive constants  $\tilde{\tau}_0:=\tilde{\tau}_0(T)$ and $\tilde{\epsilon}_0:=\tilde{\epsilon}_0(T),$ such that  when $\tau\in(0,\tilde \tau_0)$ and $\epsilon\in(0,\tilde\epsilon_0),$ 
$\sup_{t_m\in[0,T]}\||e_1(t_{m+1})|^2+|\tilde e_1(t_{m+1})|^2\|_{L^p(\Omega)}\leq C(T)\tau^2. 
$
This finishes the proof. 
\end{proof}

\section{Numerical experiments}\label{s4}
In this section, we first present some numerical experiments to verify our theoretical results on the strong and weak convergence orders. In addition, the long-time performance in calculating the ergodic limit is also illustrated by numerical experiments. 
Then we  investigate   extensions of our strategy for obtaining the second-order    method that preserve both the ergodicity and the exponential integrability, and for  obtaining the conformal symplectic method. 

\subsection{Convergence order and long-time performance}\label{sec4.1} We first consider   the strong convergence order. Take   $T=1$, $\sigma=1$, and $\upsilon=10$. 
The step-sizes of the  numerical solution for the proposed splitting methods are taken as $\tau_i=2^{-i},i=10, 11, 12, 13$. The exact solution is realized by using the same numerical scheme with
small step-size  $\tau=2^{-15}$. 
Define  error function  $
e(\tau_i,n)=\big\{\frac{1}{n}\sum\limits_{k=1}^n(|P_{T/\tau_i+1}^k-P^k(T)|^2+|Q_{T/\tau_i+1}^k-Q^k(T)|^2)\big\}^{\frac{1}{2}},
$
where $\{P_{n+1}^k,Q_{n+1}^k\}$ and $\{P^k(t),Q^k(t)\}$ are the solutions in the $k$-th sample. 
 We take $n=5000$ sample paths to simulate the expectation based on the Monte--Carlo method.   
 It is observed from  Figure \ref{Order} (a) that the mean square strong convergence order is  $1$ for SAVF, SDG, and SPAVF schemes, which verifies our theoretical result in Theorem \ref{thm:1}. 

Then we consider  the weak convergence order. Take the same parameters as above. Define  error function  
$
e^{\prime}(\tau_i,n)=\big|\frac{1}{n}\sum\limits_{k=1}^n(g(P_{T/\tau_i+1}^k,Q_{T/\tau_i+1}^k)-g(P^k(T),Q^k(T)))\big|,
$
where the test function $g(p,q)=\sin(p)\sin(q)$. 
It is observed from Figure \ref{Order} (b) that the weak convergence order of these numerical schemes is also $1$, which conforms  the result in Theorem \ref{thm:3}. 

\begin{figure}[htbp]
\centering

\subfloat[Strong  order]{
\begin{minipage}[t]{0.4\linewidth}
\centering
\includegraphics[height=4.5cm,width=5.8cm]{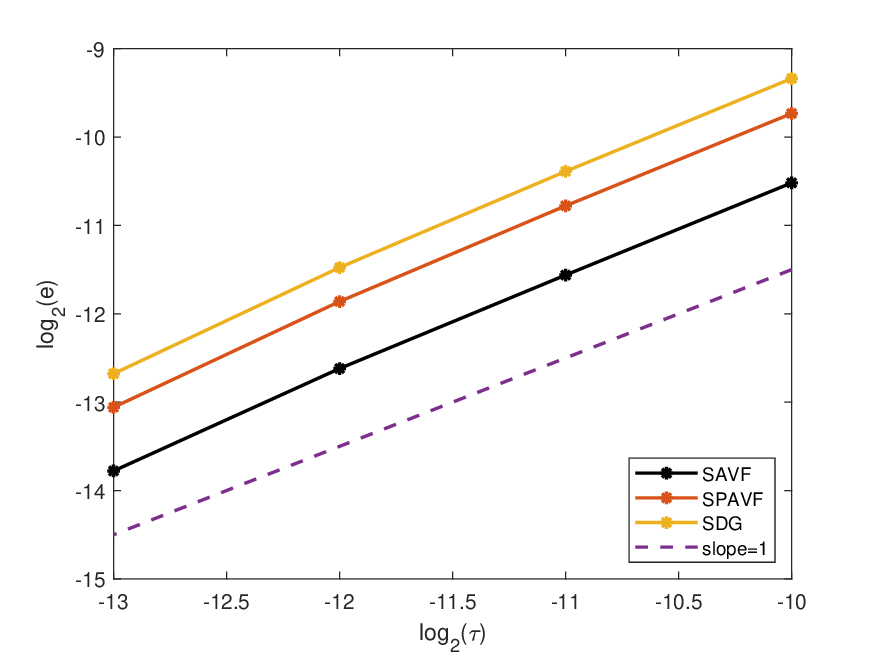}
\end{minipage}
}    \label{fig:1}
\subfloat[Weak  order]{
\begin{minipage}[t]{0.4\linewidth}
\centering
\includegraphics[height=4.5cm,width=5.8cm]{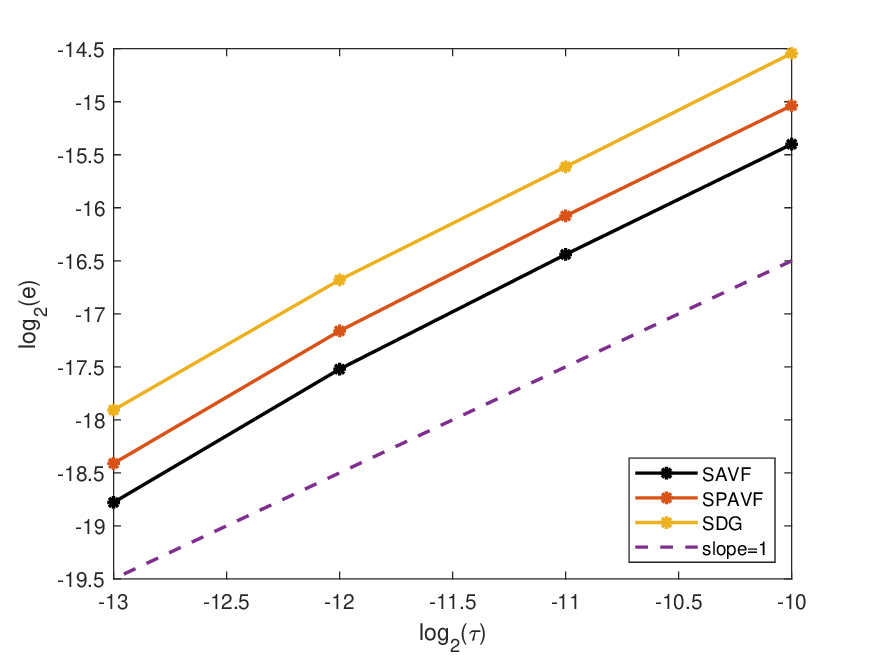}
\end{minipage}
}   \label{fig:2}

\centering
\caption{Convergence order in $\log$-$\log$ scale.}
\label{Order}
\end{figure}

Below we test the long-time performance  of the proposed methods. We take $n=1000$ sample paths for the Monte--Carlo simulation. We take    time-step  $\tau=1.25\times 10^{-6}$ to compute the reference solution, and take  time-step $\tau=10^{-5}$ to compute the numerical solution. 
In Figure \ref{Comp} (a),
 we  compute  the strong error for $t\in[0,1000]$ of the proposed  SAVF scheme,  
which illustrate  that  the proposed  scheme  has stable  strong error in long-time simulation. 
 Similar performance holds for the long-time weak  error, as shown in   Figure \ref{Comp} (b).  
   Here, we take $g(p,q)=\sin(1+(p^2+q^2)^{\frac12})$.

\begin{figure}[tbhp]
\centering
\subfloat[Strong error]{
\begin{minipage}[t]{0.4\linewidth}
\centering
\includegraphics[height=4cm,width=5.8cm]{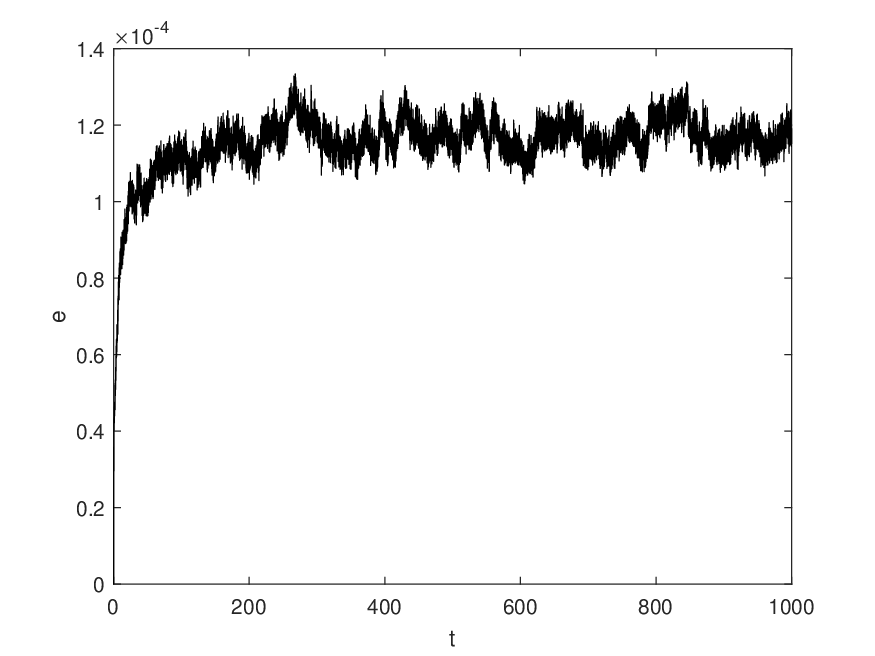}
\end{minipage}
}    \label{fig:3}
\subfloat[Weak error]{
\begin{minipage}[t]{0.4\linewidth}
\centering
\includegraphics[height=4cm,width=5.8cm]{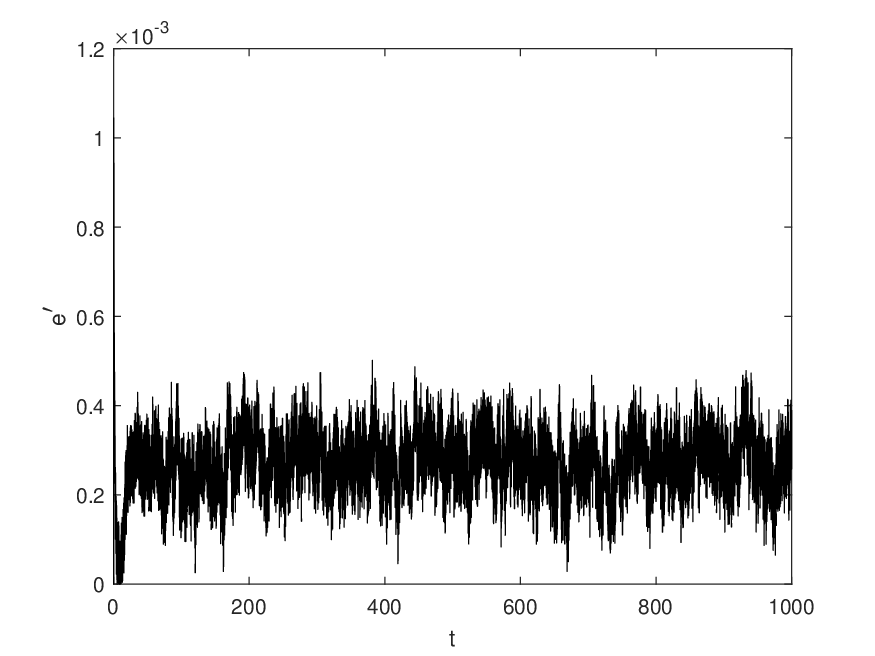}
\end{minipage}
}   \label{fig:4}
\centering
\caption{Long-time error.}
\label{Comp}
\end{figure}




To examine the numerical ergodicity for the proposed splitting methods,  we compute the empirical distribution at different time as follows. We set $\sigma =1$, $\upsilon=15$, and the initial value $P(0)=0, Q(0)=0$ and
compute over $n=5000$ sample paths till $T=512$
with $\tau=2^{-8}$ via the SAVF scheme. We
plot the empirical distribution at different time $t=0,2,256$ in Figure \ref{fig:7}, 
which shows that the empirical distribution converges to the reference Gibbs distribution   as time increases. This  indicates the ergodicity of the
numerical solution.

\begin{figure}[htbp]
  \centering
 \begin{minipage}{0.48\linewidth} 
   \vspace{3pt}
   \centerline{\includegraphics[width=\textwidth]{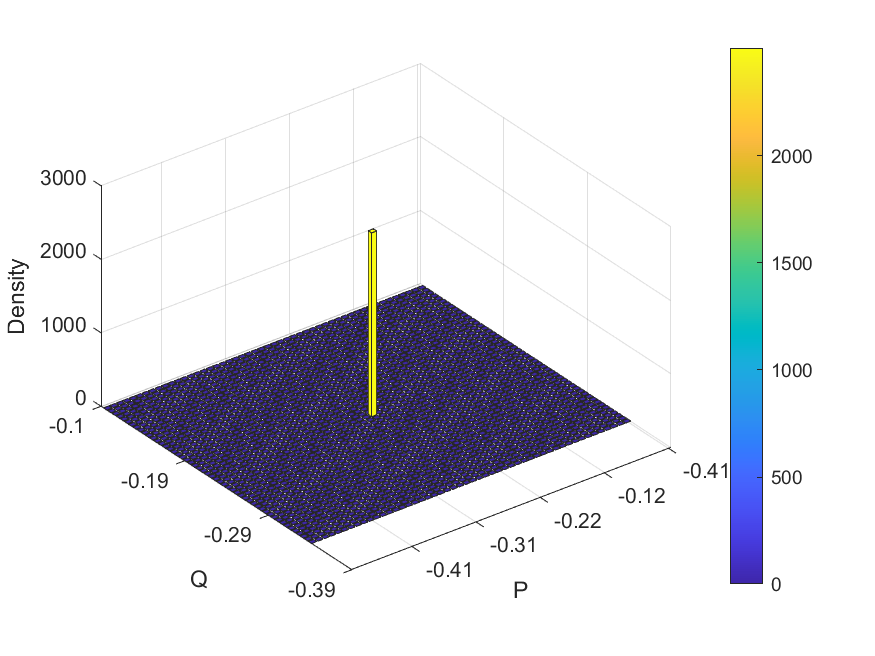}}
   \centerline{(a)}
     \centerline{\includegraphics[width=\textwidth]{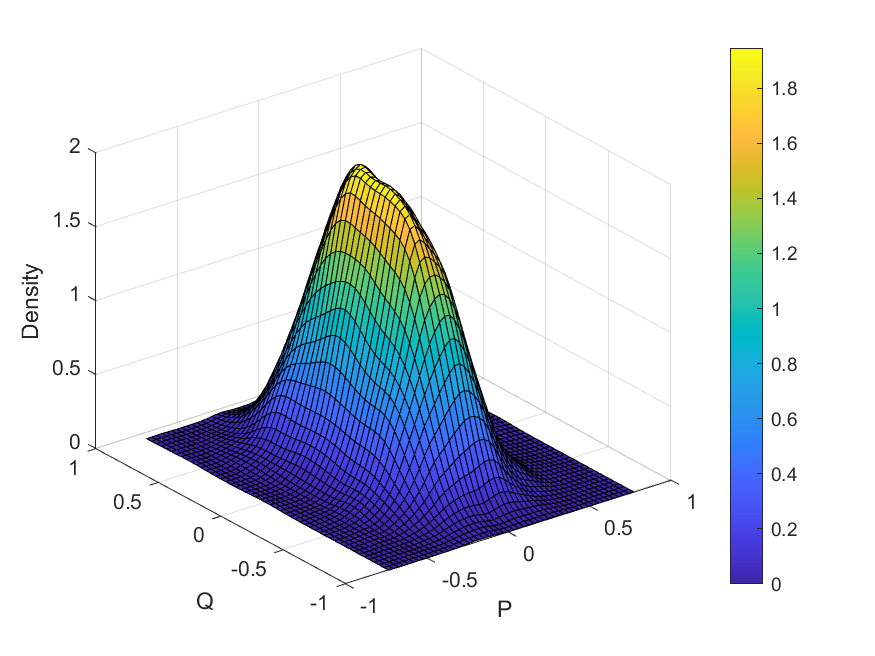}}
     \centerline{(c)}
     \vspace{3pt}
 \end{minipage}
    \begin{minipage}{0.48\linewidth}
      \vspace{3pt}
       \centerline{\includegraphics[width=\textwidth]{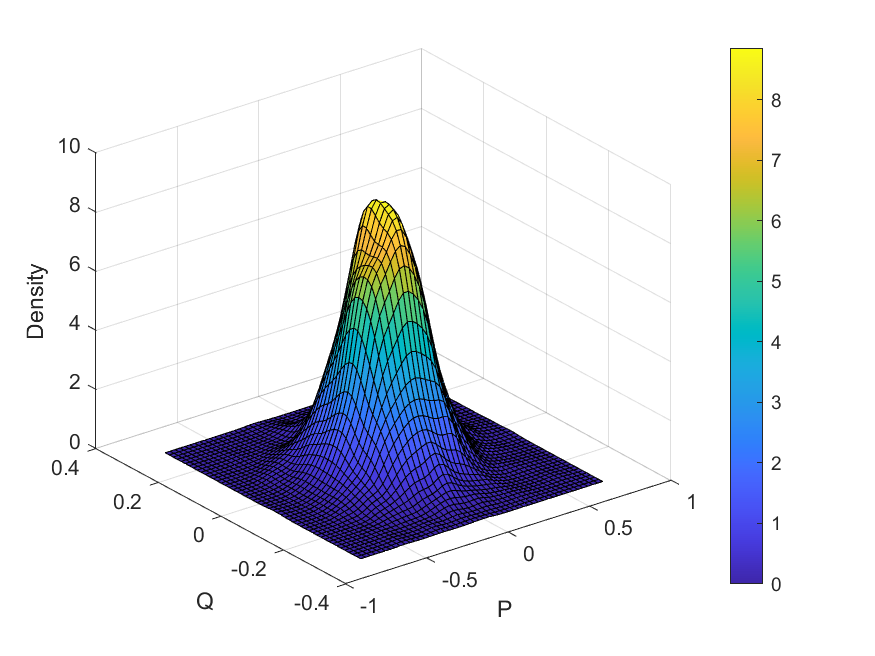}}
       \centerline{(b)}
          \centerline{\includegraphics[width=\textwidth]{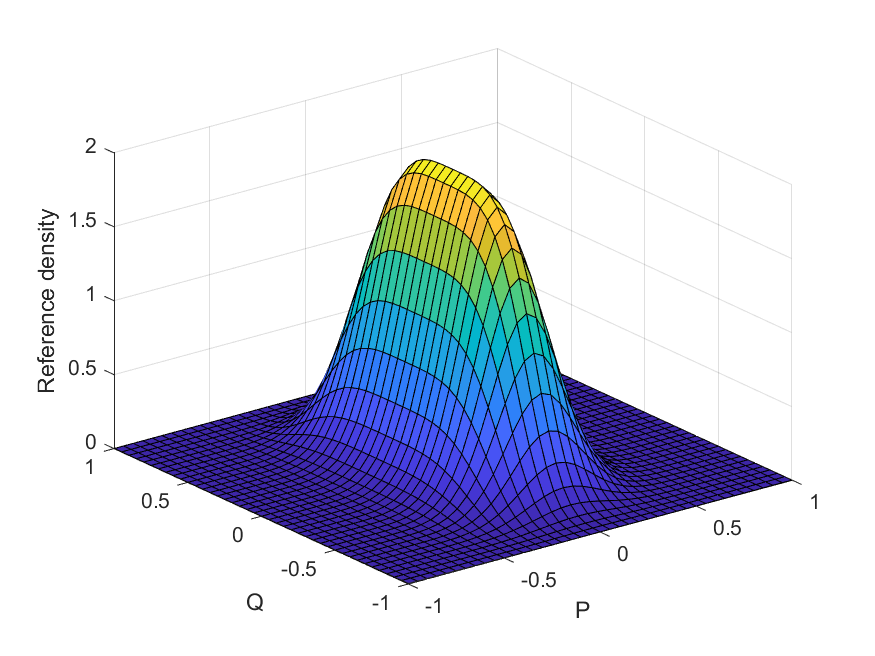}}
        \centerline{(d)}  
     \vspace{3pt}
   \end{minipage}
   	\caption{(a,b,c) The empirical distribution  at different time $t=0,2,256$,  respectively, and  (d)  the reference distribution  
$\rho(p,q)=\frac{\sqrt{\upsilon}}{\sigma\sqrt{\pi}\int_{\mathbb{R}}e^{-\upsilon q^4/(2\sigma^2)}dq}e^{-\frac{2\upsilon}{\sigma^2}(\frac{p^2}{2}+\frac{q^4}{4})}$, see \cite[page 183]{PavGri2014}.}
\label{fig:7}
 \end{figure}

      The mean square displacement of $(P(t),Q(t))$ is defined as
      \begin{equation}\label{MSD}
MSD(t) := \mathbb{E}[|(P(t),Q(t))-(P(0),Q(0))|^2],   
\end{equation}  
which  characterizes the  diffusion behavior and motion properties of molecules in a quantitative way.  It tends to an equilibrium as time grows to infinity, as shown in Figure \ref{fig:8} (a). To see the convergence rate  clearer, 
we take  time $T=512$ to simulate $MSD(\infty)$, and  compute the evolution of $MSD(\infty)-MSD(t)$  with respect to time in Figure \ref{fig:8} (b).   
It is indicated in Figure  \ref{fig:8} (b)  that the mean square displacement approaches to an equilibrium
with an exponential rate.
\begin{figure}[htbp]
 \begin{minipage}{0.48\linewidth}
   \vspace{3pt}
   \subfloat[]{
     \centerline{\includegraphics[width=\textwidth]{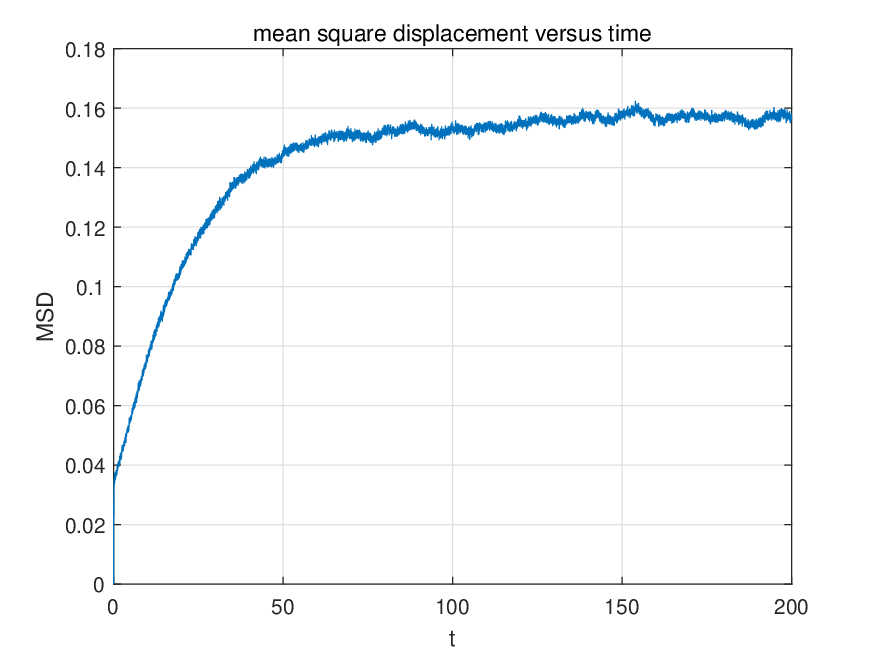}}
     }
 \end{minipage}
    \begin{minipage}{0.48\linewidth}
      \vspace{3pt}
      \subfloat[]{
        \centerline{\includegraphics[width=\textwidth]{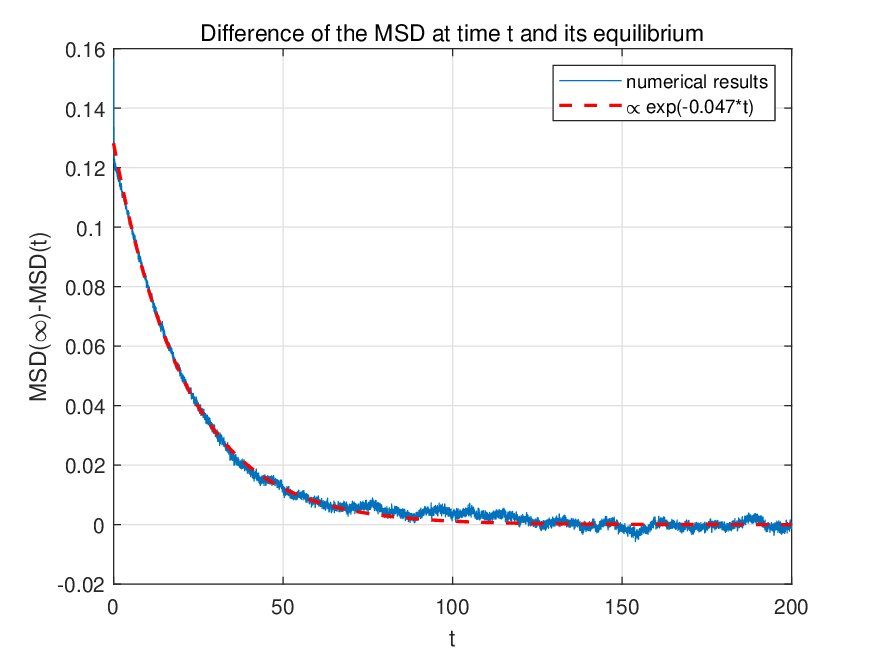}}
        }
     \vspace{3pt}
 \end{minipage}
	\caption{ (a) Mean square displacement and (b) $MSD(\infty)-MSD(t)$.}
	\label{fig:8}       
      \end{figure}

\subsection{Second-order property-preserving  method}\label{sec4.2}
Based on our strategy, one can apply the Strang splitting technique to obtain the second-order weak convergence method that preserve both the ergodicity and the exponential integrability. To be specific,  the Strang splitting technique yields the following evolution operator  \begin{align*}&\quad \mathcal P_{\tau}^{\frac12(\mathcal L_1+\mathcal L_2+\frac12(\mathcal L_3-\mathcal L'_3)),\frac12(\mathcal L_3+\mathcal L'_3)+\mathcal L_4,\frac12(\mathcal L_1+\mathcal L_2+\frac12(\mathcal L_3-\mathcal L'_3))}\\&=e^{\frac{\tau}{2}(\mathcal L_1+\mathcal L_2+\frac12(\mathcal L_3-\mathcal L'_3))}e^{\tau(\frac12(\mathcal L_3+\mathcal L'_3)+\mathcal L_4)}e^{\frac{\tau}{2}(\mathcal L_1+\mathcal L_2+\frac12(\mathcal L_3-\mathcal L'_3))}.\end{align*}
Then one can construct second-order numerical methods associated with this splitting. For example, 
we have the following  Strang SAVF scheme 
\begin{equation*}
\begin{aligned}
 &\bar{P}_{T_n,t_{n+\frac{1}{2}}} = P_n -\frac{\tau\upsilon}{8}(\bar{P}_{T_n,t_{n+\frac{1}{2}}}+P_n)-\frac{\tau}{2}\int_0^1\left(Q_n+\lambda(\bar{Q}_{T_n,t_{n+\frac{1}{2}}}-Q_n)\right)^3d\lambda,
    \\
 &  \bar{Q}_{T_n,t_{n+\frac{1}{2}}} = Q_n+\frac{\tau}{4}(\bar{P}_{T_n, t_{n+1}}+P_n)+\frac{\tau\upsilon}{8}(\bar{Q}_{T_n,t_{n+\frac{1}{2}}}+Q_n),
  \\
  &\tilde{P}_{T_n,t_{n+1}}=e^{-\frac{\upsilon \tau}{2}}\bar{P}_{T_n,t_{n+\frac{1}{2}}}+\sigma \int_{t_n}^{t_{n+1}}e^{-\frac{\upsilon}{2}(t_{n+1}-t)}dW_t,
  \\
 &\tilde{Q}_{T_n,t_{n+1}}=e^{-\frac{\upsilon \tau}{2}}\bar{Q}_{T_n,t_{n+\frac{1}{2}}},
       \\
    &P_{n+1} = \tilde{P}_{T_n,t_{n+1}}-\frac{\tau\upsilon}{8}(P_{n+1}+\tilde{P}_{T_n,t_{n+1}})-\frac{\tau}{2}\int_0^1\left(\tilde{Q}_{T_n,t_{n+1}}+\lambda(Q_{n+1}-\tilde{Q}_{T_n,t_{n+1}})\right)^3d\lambda,\\
    &  Q_{n+1} = \tilde{Q}_{T_n,t_{n+1}}+\frac{\tau}{4}(P_{n+1}+\tilde{P}_{T_n, t_{n+1}})+\frac{\tau\upsilon}{8}(Q_{n+1}+\tilde{Q}_{T_n,t_{n+1}}).
\end{aligned}
\end{equation*}
By the similar proof to that of  Proposition \ref{thm:2}, we can first show that the sequence $\{(P_{3n},Q_{3n})\}_{n\in\mathbb N_+}$ is ergodic. Then combining \eqref{pmoment0} and the preservation of energy for the AVF scheme, one can also derive \eqref{X1moment} for the considered numerical solution, and thus prove the ergodicity of the numerical solution $\{(P_n,Q_n)\}_{n\in\mathbb N_+}.$ The exponential integrability of the numerical solution can be obtained in a similar way as Proposition \ref{lem:3.1} and  thus the proof is omitted. 

We employ  the numerical experiment to illustrate the weak convergence order. Take $g(p,q)=\sin((p^2+q^2)^{\frac12})$, $n=5000$, and $T=100$. The step-sizes of the  numerical solution for the Strang   SAVF scheme are taken as $\tau_i=2^{-i},i=10, 11, 12, 13$.
 The exact solution is realized by using the same numerical scheme with
 small step-size  $\tau=2^{-15}$. It is observed from Figure \ref{Strangweak} that the weak convergence order is $2$.

\begin{figure}[htbp]
\centering
\includegraphics[height=4.5cm,width=6.8cm]{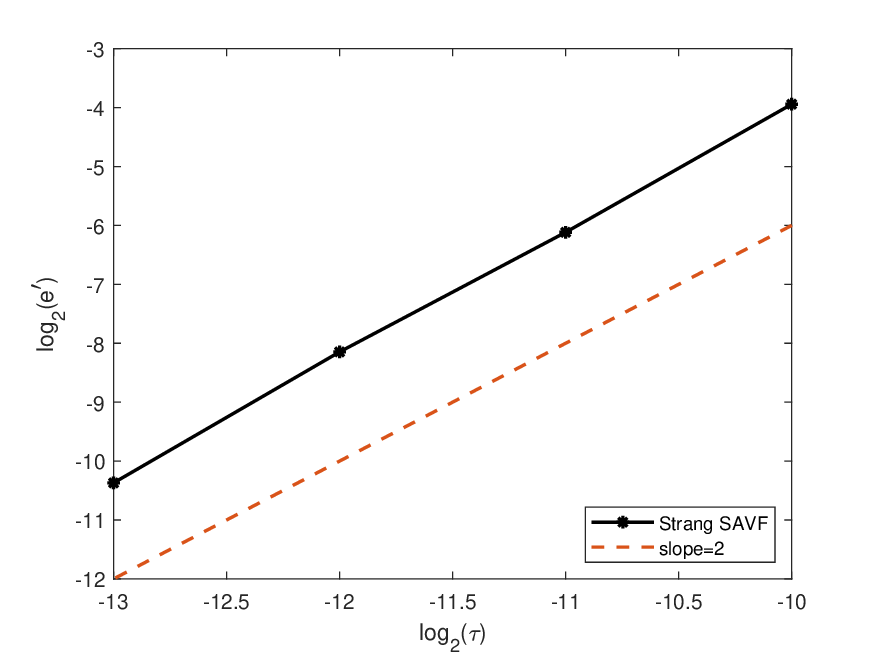}    
\caption{Weak  order for Strang SAVF scheme.}
\label{Strangweak}
\end{figure}

\subsection{Conformal symplectic splitting-based method}\label{sec4.3} When one applies the symplectic method to the  Hamiltonian subsystem \eqref{sub_ex1}, and solve the stochastic subsystem exactly, the corresponding splitting  method \eqref{SAVF} can be proved to preserve the 
 conformal symplectic structure: $\mathrm dP_{n+1}\wedge \mathrm dQ_{n+1}=e^{-\upsilon \tau}\mathrm dP_n\wedge \mathrm dQ_n,\;n\in\mathbb N.$  
To illustrate this by numerical experiment,  we use the symplectic Euler method to solve  subsystem \eqref{sub_ex1}, and obtain the one-step mapping \eqref{eq:2-1} as follows:
\begin{align*}
  &\bar{P}_{T_n,t_{n+1}}=P_n -\tau (Q_n^3+\frac{\upsilon}{2}\bar{P}_{T_n,t_{n+1}}),
  \\
   & \bar{Q}_{T_n,t_{n+1}}=Q_n+\tau ( \bar{P}_{T_n,t_{n+1}}+\frac{\upsilon}{2}Q_n),
  \end{align*}
Further solving subsystem  \eqref{sub_ex3} exactly derives the 
conformal symplectic splitting scheme \eqref{SAVF}.  Set $T=1$, $\upsilon=2$, and $\tau=10^{-4}$. We let the initial values be  in the unit circle and plot the curve of numerical solution using the above conformal symplectic splitting scheme. As shown in Figure
\ref{symplectictu} (a), the area of the curve in the phase space decreases  as time grows. We plot the logarithm of area $S(t)$ of the  curve as time $t$ increases   in Figure \ref{symplectictu} (b), which
verifies  that $S(t) = \pi exp(-\upsilon t)$.

\begin{figure}[htbp]
\centering

\subfloat[The curve in the phase space ]{
\begin{minipage}[t]{0.4\linewidth}
\centering
\includegraphics[height=4.5cm,width=5.8cm]{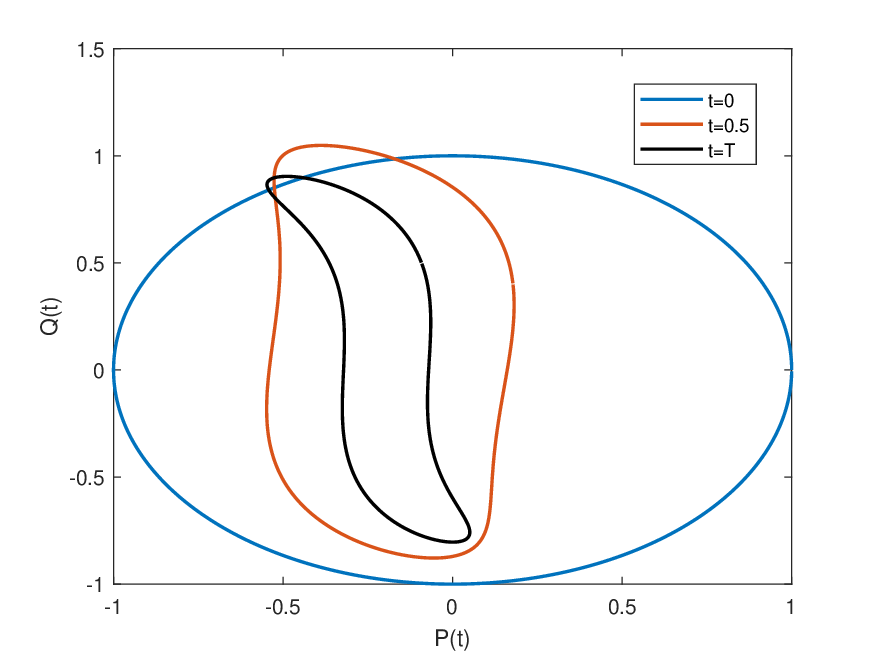}
\end{minipage}
}    
\subfloat[The logarithm of area]{
\begin{minipage}[t]{0.4\linewidth}
\centering
\includegraphics[height=4.5cm,width=5.8cm]{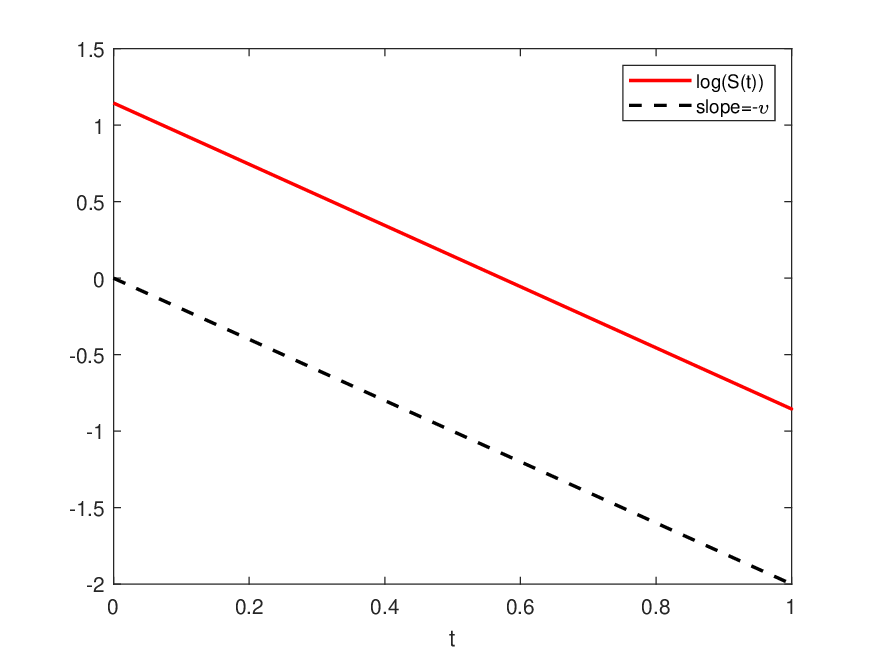}
\end{minipage}
}   
\centering
\caption{Evolution for area of domain in the phase space.}\label{symplectictu}
\end{figure}

\appendix
\section{Some proofs}\label{s2-1}

This section is devoted to presenting some auxiliary proofs, including the proofs of \eqref{exp_H} and Lemma \ref{lem:2}.
\begin{proof}[Proof of \eqref{exp_H}]
  Denote
$ {\mu}(P(t),Q(t))=(  -Q^3(t)-\upsilon P(t), P(t))^{\top}$ and $\tilde{\sigma}=(\sigma,0)^{\top}$. 
 Then one can obtain 
$
\mathcal{D}H(P(t),Q(t)){\mu}(P(t),Q(t))=-\frac{\upsilon}{2}P^2(t)-\frac{\upsilon}{2}Q^4(t)-\frac{\upsilon^2}{2}P(t)Q(t)$,
$tr(\mathcal{D}^2H(P(t),Q(t))\tilde{\sigma}\tilde{\sigma}^T)=\sigma^2,$ and $ |\tilde{\sigma}^T\mathcal{D}H(P(t),Q(t))|^2=|\sigma(P(t)+\frac{\upsilon}{2}Q(t))|^2.
$
This gives that for $\beta>0,$
\begin{align*}
 \mathcal E:= &\,\mathcal{D}H(P(t),Q(t))\mu(P(t),Q(t))+\frac{tr(\mathcal{D}^2H(P(t),Q(t))\tilde{\sigma}\tilde{\sigma}^T)}{2}+\frac{|\tilde{\sigma}^T\mathcal{D}H(P(t),Q(t))|^2}{2e^{\beta t}}\\
  \leq& -\frac{\upsilon}{2}H(P(t),Q(t))+\frac{\sigma^2}{2}H(P(t),Q(t))+\frac{\sigma^2\upsilon^4}{64}+\frac{\sigma^2}{2}.
\end{align*}
Let $\bar V=\upsilon H(P(t),Q(t))-\big(\frac{\sigma^2\upsilon^4}{64}+\frac{\sigma^2}{2}\big)$ and $\beta:=\sigma^2.$ Then 
using  \cite[Lemma
3.2]{CuiSheng2022}, we derive
\begin{align*}
\mathbb{E}\Big[\exp\Big\{\frac{H(P(t),Q(t))}{e^{\sigma^2t}}+\int_0^t\frac{\upsilon H(P(r),Q(r))-(\frac{\sigma^2\upsilon^4}{64}+\frac{\sigma^2}{2})}{e^{\sigma^2r}}dr\Big\}\Big]\leq e^{H(P_0,Q_0)},
\end{align*}
which together with \eqref{HCH} and \eqref{equ:1} yields  \begin{align*}
\mathbb E\Big[\exp\Big\{\frac{C_e(|P(t)|^2+|Q(t)|^4)}{e^{\sigma^2t}}\Big\}\Big]\leq e^{(\upsilon C_H+\frac{\sigma^2\upsilon^4}{64}+\frac{\sigma^2}{2})t+C+H(P_0,Q_0)}.
\end{align*} 
This completes the proof by taking supremum with $t\in[0,T].$
\end{proof}

\begin{proof}[Proof of Lemma \ref{lem:2}]


We use the induction argument on $p\in\mathbb N_+$ to prove \eqref{bound2} and \eqref{bound1}.  
We first prove  the case of $p=1.$ Without loss of generality, we consider $H(p,q)+C_H$ instead of $H$.  
By the It\^o formula, we have
    \begin{align*}
    dH(\tilde{P}_{T_n,t},\tilde{Q}_{T_n,t})
   & =(\tilde{P}_{T_n,t}+\frac{\upsilon}{2}\tilde{Q}_{T_n,t})d\tilde{P}_{T_n,t}+(\tilde{Q}_{T_n,t}^3+\frac{\upsilon}{2}\tilde{P}_{T_n,t})d\tilde{Q}_{T_n,t}
      +\frac{\sigma^2}{2}dt
      \\
      &\leq -\upsilon  H(\tilde{P}_{T_n,t},\tilde{Q}_{T_n,t})dt
      +\frac{\sigma^2}{2}dt+  \big(\tilde{P}_{T_n,t}+\frac{\upsilon}{2}\tilde{Q}_{T_n,t}\big)\sigma dW_t,
\end{align*}
 which yields that 
$H(\tilde{P}_{T_n,t},\tilde{Q}_{T_n,t})\leq   e^{-\upsilon (t-t_n)}H(\tilde{P}_{T_{n},t_{n}},\tilde{Q}_{T_{n},t_n})+\frac{\sigma^2}{2}\int_{t_n}^te^{-v (t-s)}ds
 +\int_{t_n}^te^{-v (t-s)}\big(\tilde{P}_{T_n,s}+\frac{\upsilon}{2}\tilde{Q}_{T_n,s}\big)\sigma dW_s.
$
By the preservation of the Hamiltonian for \eqref{eq:2-1}, we obtain $H(\tilde{P}_{T_{n},t_{n}},\tilde{Q}_{T_{n},t_n})=H(\bar {P}_{T_{n},t_{n+1}},\bar{Q}_{T_{n},t_{n+1}})=H(P_n,Q_n),$ and thus by iteration, we arrive at  
{\small\begin{align}\label{Ito1}
H(\tilde{P}_{T_n,t},\tilde{Q}_{T_n,t}) \leq  e^{-\upsilon t}H(P_0,Q_0)+\frac{\sigma^2}{2}\int_0^t e^{-\upsilon (t-s)}ds+\int_0^te^{-\upsilon (t-s)}\big(\tilde{P}_{T_{n_s},s}+\frac{\upsilon}{2}\tilde{Q}_{T_{n_{s}},s}\big)\sigma dW_s.
\end{align}} 
Applying the maximal inequality and combining \eqref{equ:1} give that for any fixed $T>0,$
{\small\begin{align*}
\mathbb E\Big[\sup_{t\in[0,T]}(|\tilde{P}_{T_{n_t},t}|^2+|\tilde{Q}_{T_{n_t},t}|^4)\Big]\leq C(H(P_0,Q_0)+1)+\int_0^T\mathbb E\Big[\sup_{r\in[0,s]}|\tilde P_{T_{n_r},r}|^2+|\tilde Q_{T_{n_r},r}|^4\Big]ds,
\end{align*}}
which together with the Gr\"onwall inequality leads to \eqref{bound2} with $p=1.$ 
 Then taking expectation on both sides of \eqref{Ito1} and using \eqref{equ:1} finish the proof for $p=1.$ 

Assume that  \eqref{bound2} and \eqref{bound1}  hold for the case of $p-1$ with $ p\ge 2.$ Then we show the case of $p.$ 
By the It\^o formula, we have
    \begin{align*}
  &  dH^p(\tilde{P}_{T_n,t},\tilde{Q}_{T_n,t})\\
      =&-\upsilon p H^p(\tilde{P}_{T_n,t},\tilde{Q}_{T_n,t})dt+p(p-1)\sigma^2H^{p-2}(\tilde{P}_{T_n,t},\tilde{Q}_{T_n,t})\Big(\frac{\tilde{P}_{T_n,t}^2}{2}+\frac{\upsilon}{2}\tilde{P}_{T_n,t}\tilde{Q}_{T_n,t}+\frac{\upsilon^2}{8}\tilde{Q}_{T_n,t}^2\Big)dt
      \\
      &+\frac{1}{2}p\sigma^2H^{p-1}(\tilde{P}_{T_n,t},\tilde{Q}_{T_n,t})dt+\sigma p H^{p-1}(\tilde{P}_{T_n,t},\tilde{Q}_{T_n,t})\Big(\tilde{P}_{T_n,t}+\frac{\upsilon}{2}\tilde{Q}_{T_n,t}\Big)dW_t.
\end{align*}
From the Young inequality, we derive 
\begin{align*}
 &dH^p(\tilde{P}_{T_n,t},\tilde{Q}_{T_n,t})
  \leq -\upsilon
pH^p(\tilde{P}_{T_n,t},\tilde{Q}_{T_n,t})dt+p(p-\frac12)\sigma^2H^{p-1}(\tilde{P}_{T_n,t},\tilde{Q}_{T_n,t})dt\\
&
+Cp(p-1)\sigma^2H^{p-2}(\tilde{P}_{T_n,t},\tilde{Q}_{T_n,t})dt+\sigma p H^{p-1}(\tilde{P}_{T_n,t},\tilde{Q}_{T_n,t})\Big(\tilde{P}_{T_n,t}+\frac{\upsilon}{2}\tilde{Q}_{T_n,t}\Big)dW_t. 
\end{align*}
Let $\gamma\in(0,1)$ be a small number to be determined later. The Young inequality gives 
\begin{align}\label{unbound1}
d (e^{\gamma pt}H^p(\tilde{P}_{T_n,t},\tilde{Q}_{T_n,t}))&\leq e^{\gamma pt}(\gamma p -\upsilon p+\epsilon p)H^p(\tilde{P}_{T_n,t},\tilde{Q}_{T_n,t})dt+C(\epsilon)e^{\gamma pt}dt\notag\\
&\quad+\sigma p H^{p-1}(\tilde{P}_{T_n,t},\tilde{Q}_{T_n,t})\Big(\tilde{P}_{T_n,t}+\frac{\upsilon}{2}\tilde{Q}_{T_n,t}\Big)e^{\gamma pt}dW_t.
\end{align}
One can choose $\epsilon$ and $\gamma$ so that $\gamma-\upsilon+\epsilon\leq -\frac{\upsilon}{2}<0.$  
Notice that  for any fixed $T>0,$ 
\begin{align*}
&\quad \mathbb E\Big[\sup_{t\in[0,T]}\Big|\int_0^tH^{p-1}(\tilde{P}_{T_{n_s},s},\tilde{Q}_{T_{n_s},s})\big(\tilde{P}_{T_{n_s},s}+\frac{\upsilon}{2}\tilde{Q}_{T_{n_s},s}\big)e^{-\gamma p(t-s)}dW_s\Big|\Big]\\
&\leq C\Big(\mathbb E\Big[\sup_{s\in[0,T]}H^{p-1}(\tilde{P}_{T_{n_s},s},\tilde{Q}_{T_{n_s},s})\Big]\Big)^{\frac12}\Big(\mathbb E\Big[\int_0^TH^p(\tilde{P}_{T_{n_s},s},\tilde{Q}_{T_{n_s},s})\Big]ds\Big)^{\frac12}+C(T)\\
&\leq C(T)+C(T)\mathbb E\Big[\int_0^TH^p(\tilde{P}_{T_{n_s},s},\tilde{Q}_{T_{n_s},s})\Big]ds. 
\end{align*}
Hence, applying  the Gr\"onwall inequality yields \eqref{bound2}. 
Furthermore, it follows from \eqref{unbound1} that  
\begin{align}\label{pmoment0}
\mathbb E[H^p(\tilde{P}_{T_n,t},\tilde{Q}_{T_n,t})]\leq e^{-\frac{\upsilon}{2}p(t-t_n)}\mathbb E[H^p(\tilde{P}_{T_n,t_n},\tilde{Q}_{T_n,t_n})]+Ce^{-\frac{\upsilon}{2}p(t-t_n)}(t-t_n). 
\end{align} 
Then combining the preservation of Hamiltonian for  \eqref{eq:2-1}, we have
\begin{align}\label{X1moment}
&\mathbb{E}[H^p(P_{n+1},Q_{n+1})]
\leq e^{-\frac{\upsilon}{2}p\tau}\mathbb{E}[H^p(P_n,Q_n)]+Ce^{-\frac{\upsilon}{2} p\tau}\tau, 
\end{align}
which yields 
$
  \mathbb{E}[H^p(P_{n+1},Q_{n+1})]\leq e^{-\frac{\upsilon}{2} np\tau}H^p(P_0,Q_0)+C.
$
Combining \eqref{equ:1} finishes the proof of \eqref{bound1}. 
\end{proof}

\bibliographystyle{plain}
\bibliography{references}

\end{document}